\documentclass[a4paper,12pt]{article}

\usepackage{amsfonts,amsthm,amsmath}

\usepackage[T1]{fontenc}

\usepackage[bookmarks,pdfpagelabels,linkbordercolor={0.1 0.8 0.4}]{hyperref}

\newtheorem{teo}{Theorem}[section]
\newtheorem*{teo*}{Theorem}
\newtheorem{lem}[teo]{Lemma}

\newtheorem{pro}[teo]{Proposition}

\theoremstyle{definition}
\newtheorem{fed}[teo]{Definition}
\theoremstyle{remark}
\newtheorem{rem}[teo]{Remark}

\def\N{\mathbb{N}}

\def\R{\mathbb{R}}
\def\C{\mathbb{C}}

\def\cB{\mathcal{B}}
\def\cC{\mathcal{C}}

\def\cF{\mathcal{F}}

\def\cH{\mathcal{H}}
\def\cI{\mathcal{I}}

\def\cK{\mathcal{K}}

\def\cL{\mathcal{L}}

\def\cO{\mathcal{O}}
\def\cP{\mathcal{P}}
\def\cQ{\mathcal{Q}}

\def\cU{\mathcal{U}}

\def\noi{\noindent}

\def\bdem{\begin{proof}}
\def\edem{\renewcommand{\qed}{\hfill $\blacksquare$}
\end{proof}}

\DeclareMathOperator{\ran}{ran}
\DeclareMathOperator{\rank}{rank}

\newcommand{\PI}[2]{\left\langle #1 , #2 \right\rangle}

\DeclareMathOperator*{\smayo}{\prec}

\title{Global symmetric approximation of frames}
\author{Eduardo Chiumiento}
\date{}

\begin{document}

\maketitle 

\begin{abstract}
\noi We solve the problem of best approximation by Parseval frames   to an arbitrary frame in a  subspace of an infinite dimensional Hilbert space. We explicitly describe all the solutions and we give a criterion for uniqueness. This best approximation problem was previously solved under an additional assumption on the set of Parseval frames in  M. Frank, V. Paulsen, T. Tiballi, {\it Symmetric approximation of frames and bases in Hilbert spaces}, Trans. Amer. Math. Soc. 354 (2002), 777-793. Our proof relies on the geometric structure of the set of all Parseval frames quadratically close to a given frame. In the process we show that its connected components can be parametrized by using the notion of index of  a pair of projections, and we prove existence and uniqueness results of best approximation by Parseval frames restricted to these connected components.   
\end{abstract}

\section{Introduction}

Let $\cH$ be an infinite dimensional separable Hilbert space. A \textit{frame} in a closed subspace $\cK \subseteq \cH$ is a sequence $\{ f_i \}$ of vectors in $\cK$ with the property that there are constants $A,B>0$ such that
$$
A \| f\|^2 \leq \sum_{i=1}^\infty |\PI{f}{f_i}|^2 \leq B \| f\|^2
$$
for all $f \in \cK$. 
 The  frame is a  \textit{Parseval frame} \textit{(or normalized tight frame)} if $A=B=1$. In 2002, Frank, Paulsen and Tiballi \cite{FPT02} 
introduced the concept of \textit{symmetric approximation of frames}. Essentially, this consists in finding the closest Parseval frame to a given frame. However, one cannot consider arbitrary Parseval frames in this best approximation problem: one has to restrict to the class of Parseval frames which are \textit{weakly similar}  to the given frame (a notion explained below). Although  symmetric approximation of frames without the weak similarity assumption has recently been considered in finite dimensions \cite{AC16}, it has never been studied in infinite dimensional Hilbert spaces without this extra assumption. This is the aim of the current paper.

Frames have been employed in signal processing in the recent decades, though their definition was conceived by  Duffin and Schaeffer \cite{DS52} to deal with problems in non-harmonic Fourier series.  Nowadays frame theory has evolved into an active area of research in mathematics with numerous applications including   image and data compression, filter banks, coding theory, optics and quantum computing  (see e.g. \cite{ CK12, K1, K2, AF97}), and it also fruitful interacts with other areas of mathematics such as operator algebras, operator theory and Banach space theory (see \cite{C00, Chrbook, HL00}).  In finite dimensional Hilbert spaces, the definition of frames reduces to the notion of spanning sets. Both notions do not coincide in infinite dimensional Hilbert spaces, but in any case frames are generally redundant and stable spanning sets, and these are the key properties which make them useful to applications.

The usual Gram-Schmidt process to find an orthonormal basis from a given set of vectors is an order-dependent method.  It is not suitable to apply with frames, where in many cases it is not clear how to order the vectors, and it also would give an orthonormal basis, loosing the redundancy property of frames. Motivated by the \textit{L\"owdin orthogonalization method} in quantum chemistry \cite{L70, AEG80}, which approximates an arbitrary basis by an orthonormal basis, the authors in \cite{FPT02} introduced the method of symmetric approximation of frames  to approximate arbitrary frames by Parseval frames.  As we mentioned above, a frame $\{ f_i\}$ in a subspace $\cK \subseteq \cH$ is a Parseval frame if satisfies Parseval's identity 
$$
\|f\|^2=\sum_{i=1}^\infty |\PI{f}{f_i}|^2
$$
for all $f \in \cK$. One of the main features of Parseval frames is the following reconstruction formula: $f=\sum_{i=1}^\infty \PI{f}{f_i}f_i$, for all $f \in \cK$, which turns out to be more complicated in the case of general frames.   
On the other hand, symmetric approximation seems to be the appropriate generalization of L\"owdin orthogonalization to the frame setting by the following well-known fact: Parseval frames can be obtained by projecting orthonormal bases from a larger Hilbert space, while arbitrary frames are given by projecting Riesz bases.




The precise definition of symmetric approximation   is as follows. First  recall that two frames $\{ f_i \}$ and $\{ g_i\}$ in subspaces $\cK$ and $\cL$, respectively, are called \textit{weakly similar} if there is an invertible bounded linear operator $T: \cK \to \cL$ such that $Tf_i=g_i$, for all $i\geq 1$. Then a Parseval frame $\{ y_i \}$ 
is said to be a \textit{symmetric approximation} of a frame $\{ f_i \}$  in a subspace $\cK \subseteq \cH$ if 
$$
\sum_{i=1}^\infty \| f_i -y_i \|^2 \leq \sum_{i=1}^\infty \| f_i -x_i \|^2
$$  
for all Parseval frames $\{ x_i \}$ in subspaces of $\cH$ which are weakly similar to $\{ f_i \}$. In order to avoid non sense comparisons the left hand of this inequality has to be finite.

We can reformulate symmetric approximation as a best approximation problem of closed range operators by partial isometries. To this end, recall that the \textit{synthesis operator}  of a frame $\{ f_i\}$ in $\cK\subseteq \cH$ is the bounded operator defined by
$$
F: \ell^2 \to \cH, \, \, \, F(e_i)=f_i \, , \, \, \, i \geq 1,
$$  
where $\{ e_i\}$ is the standard basis of $\ell^2$. This operator obviously has range $\ran(F)=\cK$. The synthesis operator establishes  a correspondence between frames in subspaces of $\cH$ and closed range operators. In particular, the synthesis operator of a Parseval frame is a partial isometry. Denote by $\| \, \cdot \, \|_2$ the usual Hilbert-Schmidt norm. Then a Parseval frame $\{ y_i \}$  with synthesis operator $Y$
is  a symmetric approximation of a frame $\{ f_i \}$   in a subspace $\cK \subseteq \cH$ with synthesis operator $F$ if the estimate
$$
\| F- Y \|_2 \leq \| F - X\|_2
$$
is valid for all partial isometries $X$ such that $\ker(X)=\ker(F)$. 
This latter condition on the kernels is clearly equivalent to the weak similarity of the frames. 
Now we can describe the main result on symmetric approximation in \cite{FPT02}. Let $F=U|F|$ be the  polar decomposition (i.e. $U$ is the unique partial isometry satisfying $F=U|F|$ and $\ker(U)=\ker(F)$). Then the Parseval frame $\{u_i\}=\{Ue_i\}$ is the unique symmetric approximation of $\{ f_i\}$, whenever the condition $\sum_{i=1}^\infty \| f_i - u_i \|^2< \infty$ is verified. 
  This frame $\{ u_i\}$ is called the \textit{canonical Parseval frame} associated to $\{ f_i\}$, and satisfies that  $\overline{\text{span}}\{ u_i \}=\cK$. 

In addition, we point out that best approximation of operators by partial isometries with several operators norms was studied in  \cite{M89, W86}; meanwhile  for previous results on  best approximation of frames with specific structure by Parseval frames  we refer to  Janssen and Strohmer \cite{JS02},   Feichtinger,  Kozek and  Luef \cite{KLF09}, and Han \cite{H1, H2}. 


\bigskip

\noi \textbf{The results of this paper.}  In this paper we address the following  questions in the setting of infinite dimensional Hilbert spaces:  Do we have a  closest Parseval frame to a given frame when the assumption of weak similarity is dropped? When do we have a unique closest Parseval frame? To emphasize that the assumption  of weak similarity is removed, we shall often use the term \textit{global symmetric approximation}. Simple examples can be constructed to show that the canonical Parseval frame may not be a global symmetric approximation  (see  \cite{FPT02, AC16}).   

 For our purpose  we first get in Section \ref{the canonical Parseval} a better understanding of  the set of all Parseval frames quadratically close to a given frame (see e.g. \cite{Ch95}). Recall that two frames $\{ f_i\}$ and $\{ g_i\}$ in subspaces $\cK$ and $\cL$, respectively,  are said to be \textit{quadratically close} if 
$$
\sum_{i=1}^\infty \| f_i - g_ i\|^2< \infty.
$$
Now fix a frame $\cF=\{ f_i\}$ in a subspace $\cK\subseteq \cH$ with synthesis operator $F$, and let $\{ u_i\}$ be its canonical Parseval frame. Denote by $\cQ_\cF$ the set of all Parseval frames quadratically close to $\{ f_i\}$. 
We show that the condition 
$\sum_{i=1}^\infty \| f_i - x_ i\|^2< \infty
$ for some Parseval frame $\{ x_i \}$, implies that the synthesis operator $F$ admits a singular value decomposition.  
Based on this fact, we prove our first main result: the set $\cQ_\cF$ is nonempty  if and only if $\{ u_i\}$ is quadratically close to $\{ f_i\}$  (Theorem \ref{cuadratically close}). 
A direct consequence is  that the notion of global symmetric approximation of frame $\{ f_i\}$ only makes sense when 
$
\sum_{i=1}^\infty \| f_i - u_ i\|^2< \infty.
$
Note that under this assumption,  $\cQ_\cF$ can also be described as the set of all Parseval frames quadratically close to $\{ u_i\}$. 

The strategy in \cite{AC16} for carrying out global symmetric approximation  in finite dimensional spaces heavily relies on the geometry of the partial isometries. This set has a finite number of connected components parametrized by the rank, and it is possible to find symmetric approximations restricted to each component. Due to the finite dimension of the underlying space, a global symmetric approximation exists, and must be one of the symmetric approximations previously found from the connected components. Returning to the infinite dimensional setting, obvious difficulties arise: the rank does not determine the connected components of the partial isometries, and a priori the existence of global symmetric approximations is not guaranteed. 

However, the following simple observation allows us to follow similar steps: given $\{ x_i \} \in \cQ_\cF$ a Parseval frame in $\cL \subseteq \cH$ with synthesis operator $X$, and the canonical Parseval frame $\{ u_i \}$ associated to $\cF$ with synthesis operator $U$, then the operator
$$
XU^*|_{\cK}:\cK \to \cL
$$ 
is Fredholm. Equivalently, the pair of projections $(U^*U,X^*X)$ is a Fredholm pair  \cite{AS94, ASS94}. The index of the operator $XU^*|_{\cK}:\cK \to \cL$ coincides with the index of the mentioned pair of projections, and it determines the connected components of the set $\cQ_\cF$. It is not difficult to see that there are always infinitely many connected components. Furthermore, each connected component can be described as an orbit of an action by a  Banach-Lie group. These geometric results are presented in Section \ref{components qf}.

The question of whether global symmetric approximation exists requires an interesting intermediate step. We study  symmetric approximation by Parseval frames restricted to the connected components of $\cQ_\cF$. This can be thought as  \textit{local  symmetric approximation}. 
Now we briefly  describe the results in this regard; see Section \ref{symm connected comp} for the precise statements and details.  The canonical Parseval frame is the unique symmetric approximation  belonging to the zero index component (Theorem \ref{canonico}). Since this connected component strictly contains the set of all Parseval frames weakly similar to $\cF$, our result extends the previously mentioned result on symmetric approximation from \cite{FPT02}. Symmetric approximations from negative index components also exist, and there are always infinitely many (Theorem \ref{S aprox}).  In contrast,  symmetric approximations restricted to a  component of index $k>0$ exist if and only if the  synthesis  
operator of the frame $\cF$ has at least $k$  lowest singular values. In this case, when symmetric approximations exist, it turns out that there are only two possibilities: there is a unique approximation or infinitely many of them (Theorem \ref{k>0 aprox}).  
In the aforementioned results, all the symmetric approximations restricted to any  connected component are completely described. 

In Section \ref{section global}, we use our previous results to eventually prove that  global symmetric approximations always exist whenever $\cQ_\cF \neq \emptyset$. 
Indeed, we  give an explicit formula for global symmetric approximations (Theorem \ref{global thm}).
The  singular value decomposition of the synthesis operator of  the frame $\cF$, allows us to determine the connected component where a global symmetric approximation belongs. We shall see that this depends on the number of nonzero singular values placed below $1/2$. We also give a necessary and sufficient criterion for uniqueness of global symmetric approximation in terms of the singular values.

\section{The canonical Parseval frame}\label{the canonical Parseval}

In this section we prove the following:

\medskip

\begin{teo}\label{cuadratically close}
Let $\{ f_i \}$ be a frame in a subspace $\cK$ of a Hilbert space $\cH$. The following are equivalent:
\begin{enumerate}
\item[i)] There is a Parseval frame $\{ x_i \}$ in a subspace $\cL \subseteq \cH$ such that
$$ \sum_{i=1}^\infty \| f_i - x_i\|^2 < \infty.$$
\item[ii)] The canonical Parseval frame $\{  u_i \}$ associated to $\{  f_i \}$ satisfies
$$ \sum_{i=1}^\infty \| f_i - u_i\|^2 < \infty.$$
\end{enumerate}
\end{teo}

Before proving this, we need to give some preliminary notions and results. The proof will rely on the Lidskii-Mirsky-Wielandt theorem, an approximation argument using the index of a pair of projections and the representation of the synthesis operator associated to $\{ f_i \}$ given in Lemma \ref{spectral}. 

\medskip

\noi \textit{Majorization and Lidskii-Mirsky-Wielandt theorem.}
Let $x=(x_1, \ldots, x_n)$ be a vector in $\R^n$. We denote by $x^{\downarrow}$ the vector obtained by rearranging
the coordinates of $x$ in  nonincreasing order. This means that  
$x^{\downarrow}=(x_1^{\downarrow}, \ldots, x_n^{\downarrow})$ satisfies $x^{\downarrow}_1\geq \ldots \geq x^{\downarrow}_n$.
Given two vectors $x,y \in \R^n$, we write $x \smayo_w  y$ if 
\begin{equation}\label{def may}
\sum_{i=1}^k x^{\downarrow}_i \leq \sum_{i=1}^k y^{\downarrow}_i , \, \, \, \, \, \, \, \,  k=1, \ldots,n. 
\end{equation}
In this case, we say that $x$ is \textit{submajorized} by $y$ and we write $x\smayo_w y$. If, in addition, there is an equality for $k=n$ in (\ref{def may}), then we write  $x \smayo y$ and  we say that $x$ is \textit{majorized} by $y$. 

Given an $n \times n$ matrix $X$,  we denote by $s(X)=(s_1(X), \ldots, s_n(X))$ the vector of its singular values  arranged in non increasing order.   The Lidskii-Mirsky-Wielandt theorem states (see e.g. \cite{B97, LM99}):

\begin{teo}\label{mirsky}
Let $F,G$ be $n\times n$ matrices,  then
$$
(|s_1(F)-s_1(G)|, \ldots , |s_n(F)-s_n(G)| ) \,   \textstyle{\smayo_w}  \, s(F-G).
$$
\end{teo}

\medskip

\noi \textit{Index of a pair of projections.} 
We  recall the notion of index of a pair of projections (see \cite{AS94, ASS94}). Let $P,Q$ be two orthogonal projections on a Hilbert space $\cH$. The pair $(P,Q)$ is called a \textit{Fredhom pair} if the operator $QP|_{\ran(P)}:\ran(P) \to \ran(Q)$ is Fredholm. In this case, the \textit{index of the pair (P,Q)} is defined by 
\begin{align*}
j(P,Q)& := \text{index}(QP|_{\ran(P)}:\ran(P) \to \ran(Q)) 
\\ & =\dim(\ker(Q) \cap \ran(P)) - \dim(\ran(Q)\cap \ker(P)).
\end{align*}

\begin{rem}\label{comp and sum} We now state two elementary properties of Fredholm pairs which will be used throughout this paper. The first one follows easily, and the second was proved in \cite[Thm. 3.4]{ASS94}.
\begin{enumerate}
\item[i)] If the operator $P-Q$ is compact, then $(P,Q)$ is a Fredholm pair.
\item[ii)] If $(P,Q)$ and $(Q,R)$ are Fredholm pairs, and either $Q-R$ or $P-Q$ are compact, then $(P,R)$ is a Fredholm pair and $j(P,R)=j(P,Q)+j(Q,R)$.
\end{enumerate}
\end{rem}


Let $\cU(\cH)$ be the group of unitary operators on $\cH$. We consider the Fredholm unitary group, i.e.
$$
\cU_c(\cH)=\{  \,  L \in \cU(\cH) \, : \, L- I \text{ is a compact operator} \, \}.
$$ 
It is a  Banach-Lie group in the  topology defined by  the usual operator norm.
 In this section, we will use the following relation between pairs of projections and the Fredholm unitary group proved by Carey \cite[Remark 3.7]{C85}.

\begin{teo}\label{carey}
Let $P,Q$ be orthogonal projections on a Hilbert space $\cH$. Then $P-Q$ is compact and $j(P,Q)=0$ if and only if 
there is an unitary $L \in \cU_c(\cH)$ such that  $Q=LPL^*$. 
\end{teo}

\begin{rem}\label{holds for p2} It is worth pointing out that this result was proved in a more general setting, the compact operators can indeed be replaced by any symmetrically-normed ideal in the sense of \cite{GK60}. Throughout this paper, we shall also use this result for the ideal of Hilbert-Schmidt operators. We  refer to \cite{SV78} for an earlier proof of this case.
\end{rem}

\medskip

 Denote by $F$ and $X$  the synthesis operators of the frames $\{ f_i\}$ and $\{ x_i\}$ in the statement of Theorem \ref{cuadratically close}.
The fact that these frames are quadratically close, or equivalently, $F-X$ is a Hilbert-Schmidt operator, determines that $F$ admits the following singular value decomposition. Below we write $f\otimes g$ for the rank one operator defined by $(f\otimes g)h=\PI{h}{g}f$, where $f,g,h \in \cH$. 

\begin{lem}\label{spectral}
Let $F$ be an operator with closed range, $F=U|F|$ its polar decomposition   and $X$ a partial isometry. If  the operator $F-X$ is  compact, then  there exists an orthonormal basis $\{ g_i \}_1^\alpha$ of $\ker(F)^\perp$  satisfying
\begin{equation}\label{desc}
F=\sum_{i=1}^\alpha s_i (Ug_i) \otimes g_i \, , \, \, \, \, 1 \leq \alpha \leq \infty,
\end{equation}
where $\{ s_i\}_1^\alpha$ are the nonzero singular values of $F$. In the case where $\alpha=\infty$,  the series is unconditionally convergent in the weak operator topology and $s_i \to 1$.
\end{lem}
\begin{proof}
 Since $F -X$ is a compact operator,  $F^*F - F^*X$ and $F^*X -X^*X$ are also compact operators. Thus, $|F|^2-Q$ is compact, where $Q=X^*X$. Therefore the essential spectrum of $|F|^2$ satisfies $\sigma_{ess}(|F|^2)=\sigma_{ess}(Q)\subseteq \{ 0,\,1\}$. This gives also $\sigma_{ess}(|F|)\subseteq \{ 0, \, 1\}$. The operator $F$ has closed range, then $F^*$ also has this property, and thus $\ran(|F|)=\ran(F^*)$ is closed. Consequently, there are two possibilities:  $0$ does not belong to the spectrum  of $|F|$ or $0$ is an isolated point of the spectrum. Hence the spectrum of $|F|$ must be equal to $\{ s_i \}_1 ^\alpha \cup \{ 0\}$ or $\{ s_i \}_1 ^\alpha$, where each $s_i$ is a nonzero singular value of $F$, $1 \leq \alpha \leq \infty$ and $s_i \to 1$ when $\alpha=\infty$. By the condition on the essential spectrum, each $s_i\neq 1$ is an eigenvalue with finite multiplicity of the operator $|F|$. 

Applying the spectral theorem for self-adjoint operators, we find that there is an orthonormal basis $\{g_i\}_1 ^\alpha$ of $\ker(|F|)^\perp$  such that
$$
|F|= \sum_{i=1}^{\alpha} s_i \,  g_i \otimes g_i \, ; \, \, \, 1 \leq \alpha  \leq \infty.
$$
When $\alpha=\infty$, this series converges in the weak operator topology. Using  the polar decomposition, we have 
$$
F=U|F|=\sum_{i=1}^{\alpha} s_i \,  (Ug_i) \otimes g_i \, .
$$
  The convergence only depends on the fact that $\{ g_i \}_1^\alpha$ and $\{ Ug_i\}_1^\alpha$ are orthonormal bases of $\ker(F)^\perp$ and $\ran(F)$, respectively, which implies that the series is unconditionally convergent when $\alpha=\infty$. 
\end{proof}

We are now in position to prove Theorem \ref{cuadratically close}.

\begin{proof}
The nontrivial assertion is $i) \Rightarrow ii)$. Let $F$, $X$ and $U$ be the synthesis operators of $\{ f_i\}$, $\{ x_i\}$ and $\{ u_i\}$, respectively. We begin by noting that $F -U$ is a compact operator. 
 Indeed, by Lemma \ref{spectral},
$$
F-U=\sum_{i=1}^{\alpha} (s_i-1) \,  (Ug_i) \otimes g_i.
$$
Since $s_i \to 1$ when $\alpha=\infty$, it follows  that $F-U$ is a compact operator. Then the  sequence of its singular values arranged in nonincreasing order $\{ s_i(F-U)\}$ is well defined. 

Next note that the frames $\{ f_i \}$ and $\{ u_i \}$ are quadratically close if and only if the operator $F-U$ is Hilbert-Schmidt. This latter condition will follow immediately from the fact that $F-X$ is Hilbert-Schmidt and the  dominance property (see e.g. \cite[p. 82]{GK60}), that is 
\begin{equation}\label{dominance property}
\sum_{i=1}^m s_i(F-U) \leq\sum_{i=1}^m s_i(F-X), \, \, \, \, \forall \, m \geq 1. 
\end{equation}
Our task now is to prove these inequalities for the singular values.

\medskip

Let $P$ be the orthogonal projection onto $\ker(F)^{\perp}$. From the proof of Lemma \ref{spectral}, we have that $|F|^2-Q=K$ is compact, and
\begin{align*}
P-Q & =   P - |F|^2 + K\\
& =  \sum_{i=1}^{\alpha}  g_i \otimes g_i   - \sum_{i=1}^{\alpha} s_i ^2 \,  g_i \otimes g_i  + K\\
& = \sum_{i=1}^{\alpha} (1-s_i ^2 ) \,  g_i \otimes g_i  + K.
\end{align*}
Using again that $s_i \to 1$ when $\alpha = \infty$, it follows that $P-Q$ is compact. In particular, this gives that $(P,Q)$ is a Fredholm pair. 

If $0<k=j(P,Q)$, then define the block diagonal operators 
$$
\tilde{F}: \ell^2 \oplus \C^k \to \cH \oplus \C^k, \, \, \tilde{F}=F \oplus 0, \hspace{0.8 cm}   \tilde{X}: \ell^2 \oplus \C^k \to \cH \oplus \C^k, \, \, \tilde{X}=X\oplus I\, .
$$ 
The corresponding projections $\tilde{P}=P\oplus 0$ and $\tilde{Q}=Q\oplus I$ now satisfy that $\tilde{P}-\tilde{Q}$ is compact and $j(\tilde{P},\tilde{Q})=0$.  Accordingly, the canonical partial isometry is given by $\tilde{U}=U\oplus 0$. In what follows, we will prove that 
$\tilde{F}- \tilde{U}$ is Hilbert-Schmidt by using that $\tilde{F}-\tilde{X}$ is Hilbert-Schmidt, but this is equivalent  to have that $F-U$ is Hilbert-Schmidt.  Following a similar argument, we can modify $F$ and $X$ to obtain a Fredholm pair with zero index in the case where $j(P,Q)<0$. In the remainder of the proof, for short we write  $F$, $X$ and $U$ to refer to the operators $\tilde{F}$, $\tilde{X}$ and $\tilde{U}$, and by the preceding remarks, we can assume that $j(P,Q)=0$.

By Theorem \ref{carey}, there exists a unitary operator $L \in \cU_c(\ell^2)$ such that $Q=LPL^*$. In the case where $\alpha <\infty$, $F$ has finite dimensional range, and thus $F-U$ is Hilbert-Schmidt. So we assume that $\alpha=\infty$.  For $n \geq 1$, set
$$
P_n:= \sum_{i=1}^n  g_i \otimes g_i  \, ; \, \, \, \, \, \, Q_n:=LP_nL^*.
$$
Note that $\rank(P_n)=n$, $P_n\leq P$, $Q_n\leq Q$ and $P_n \to P$, $Q_n \to Q$ strongly.
On one hand, note that 
$$
|(F-U)P_n|= \sum_{i=1}^n |s_i-1| \, g_i \otimes g_i \, ,
$$
which gives
\begin{equation*}\label{first sing v}
s((F-U)P_n)=(|s_1-1|, \ldots, |s_n -1|)^{\downarrow}.
\end{equation*}
On the other hand, 
\begin{equation*}\label{second sing v}
s(FP_n)=(s_1, \ldots , s_n)^{\downarrow},
\end{equation*}
and since $Q_n \leq Q$, 
\begin{equation*}\label{third sing v}
s(XQ_n)=(1, \ldots , 1).
\end{equation*}
It follows that
$$
s((F-U)P_n)=|s(FP_n) - s(XQ_n)|^{\downarrow}, \, \, \, \, n \geq 1.
$$
Now if $m\geq 1$, then by Theorem \ref{mirsky} we get 
\begin{align*}
\sum_{i=1}^m s_i((F-U)P_n) & = \sum_{i=1}^m |s(FP_n)-s(XQ_n)|_i ^{\downarrow} \\
& \leq \sum_{i=1}^m s_i(FP_n-XQ_n)
\end{align*}
In the terminology of \cite{GK60}, this can be rewritten using the Ky-Fan norms $\| \, \cdot \, \|_{(m)}$ given by the sum of the first $m$ largest singular values:
\begin{equation}\label{Ky Fan}
\| (F-U)P_n \|_{(m)} \leq \|FP_n - XQ_n \|_{(m)}\, .
\end{equation}
Note that $(F-U)P_n \to F-U$ and $FP_n -XQ_n \to F-X$ strongly. Using that $F-U$ is compact, it follows by \cite[Thm. 6.3]{GK60} that
$\|(F-U)P_n - (F-U)\|_{(m)}\to 0$. A similar argument can be applied to get the convergence in the Ky-Fan norm of the right-hand side of \eqref{Ky Fan}, since
\begin{align*}
FP_n - XQ_n  & = FP_n- XLP_nL^* =FP_n - X(L-I)P_nL^* - XP_nL^* \\
& =(F-X)P_n  - X(L-I)P_nL^* - XP_n(L^*-I),
\end{align*}
and $L-I$, $L^*-I$ and $F-X$ are compact operators. Letting $n \to \infty$ in the inequality \eqref{Ky Fan}, we obtain
$$
\| F-U\|_{(m)} \leq \| F- X \|_{(m)}
$$
for all $m \geq 1$. Thus, we have proved the dominance property in \eqref{dominance property}. This finishes the proof. 
\end{proof}

\section{Parseval frames quadratically close to a frame}\label{components qf}

Let $\cF=\{ f_i\}$ be a frame in a subspace $\cK \subseteq \cH$. Let $\cQ_{\cF}$ denote the set  of all Parseval frames in subspaces of $\cH$ which are quadratically close to $\{ f_i\}$, i.e. 
$$
\cQ_{\cF}=\left\{ \, \{ x_i\} \text{ is a Parseval frame in } \cL   \, : \, \cL  \subseteq \cH , \, \sum_{i=1}^\infty\| x_i - f_i \|^2 < \infty \,  \right\}.
$$
Of course, we may have $\cQ_\cF = \emptyset$. In order to study its geometric structure,  we shall implicitly assume that $\cQ_\cF$ is a nonempty set for the remainder of this section.  Recall that by  Theorem \ref{cuadratically close},  $\cQ_\cF\neq \emptyset$ if and only if the canonical Parseval frame  $\{ u_i \}$ is quadratically close to $\{ f_i \}$. Therefore we have the following alternative description
\begin{equation}\label{descrip 2}
\cQ_{\cF}=\left\{ \, \{ x_i\} \text{ is a Parseval frame in } \cL   \, :  \, \cL \subseteq \cH , \, \sum_{i=1}^\infty\| x_i - u_i \|^2 < \infty \, \right\}.
\end{equation}
It is  natural to endow $\cQ_\cF$ with the $\ell^2$-distance:
$$
d(\{ x_i\}, \{ y_i\})=\left(\sum_{i=1}^\infty \| x_i - y_i\|^2\right)^{1/2}, \,\, \, \, \{ x_i\}, \{ y_i\} \in \cQ_{\cF}.
$$
Clearly, $\cQ_\cF$ becomes a metric space with this distance. 
We now turn our attention to the problem of characterizing  the connected components of $\cQ_\cF$, which will later play a fundamental role in  symmetric approximation.

\begin{rem}\label{homeo affine}
Given a Parseval frame $\{ x_i \}$ in a subspace of $\cH$, whose synthesis operator is a partial isometry  $X:\ell^2 \to \cH$, note that $\{ x_i \} \in \cQ_\cF$ if and only if the operator $X-F$  is Hilbert-Schmidt. Using the description in \eqref{descrip 2}, this is equivalent to have that $X-U$ is a Hilbert-Schmidt operator, where $F=U|F|$ is the polar decomposition.  Denote by $\cI$ the set of partial isometries from $\ell^2$ to $\cH$ and $\cB_2(\ell^2,\cH)$ the space of Hilbert-Schmidt  operators between these Hilbert spaces. Then the map 
$$\cQ_\cF \to \cI \cap (U+ \cB_2(\ell^2,\cH) ), \, \, \, \, \, \{ x_i\}=\{ Xe_i \} \mapsto X,$$ 
is a  homeomorphism. Here $\{ e_i\}$ is the standard basis of $\ell^2$ and the  space $\cI \cap (U+ \cB_2(\ell^2,\cH))$ is endowed with the metric given by the  Hilbert-Schmidt norm $d_2(X,Y)= \| X-Y\|_2$.
 Throughout this section, to prove properties of the Parseval frames in $\cQ_\cF$ we will often change to  partial isometries in $\cI \cap (U+ \cB_2(\ell^2,\cH))$. 
\end{rem}

Let $\cU_2(\cH)$ be the Banach-Lie group consisting of unitaries which are Hilbert-Schmidt perturbations of the identity on a Hilbert space $\cH$, i.e.
\begin{equation*}
\cU_2(\cH):=\{  \, L \in \cU(\cH) \, : \, L -I \text{ is Hilbert-Schmidt}  \, \}.
\end{equation*}
The product group $\cU_2(\cH) \times \cU_2(\ell^2)$  acts on $\cQ_\cF$ as follows: 
\begin{equation}\label{action uxu}
 (V,W)\cdot \{ x_i\} = \{ VXW^*e_i \},   \, \, \, \, W \in \cU_2(\ell^2), \, V \in \cU_2(\cH),
\end{equation}
Indeed, note that if  $X-U$ is Hilbert-Schmidt, then  
$VXW^*-U= (V-I)XW^* + X(W^*-I) + X-U$ is Hilbert-Schmidt. Given $\{ x_i\} \in \cQ_\cF$, its orbit is given by
$$
\cO(\{ x_i\})=\{ \, \{ VXW^*e_i\} \, : \, V \in \cU_2(\cH), \, W \in \cU_2(\ell^2)  \, \}.
$$
Our next results rely on the following spatial characterization of these orbits (see \cite[Thm. 2.4]{Ch10}). We switch to partial isometries as explained in Remark \ref{homeo affine}.

\begin{lem}\label{stiefel man}
Consider $X,Y:\ell^2 \to \cH$  two partial isometries. Then  $X-Y$ is Hilbert-Schmidt and $j(X^*X,Y^*Y)=0$ if and only if 
 there exist $W \in \cU_2(\ell^2)$ and $V \in \cU_2(\cH)$ such that $X=VYW^*$.
\end{lem}

 We now give the main result of this section. 


\begin{pro}\label{compo con}
Let $\cF=\{ f_i\}$ be a frame in a subspace $\cK \subseteq \cH$. Let $F$ be its synthesis operator and $F=U|F|$  the polar decomposition.
Set $n_1:=\dim(\ker(F))$, $n_2:=\dim(\ker(F)^\perp)$ and $n_3:=\dim(\ran(F)^\perp)$. Then $\cQ_\cF$ can be expressed as the union of infinitely many connected components labeled as follows: 
$$
\cQ_\cF=\bigcup_{k \in \mathbb{J}}\cQ_\cF^k\, , \, \, \, \, \, \, \, \, \mathbb{J}=[- \min \{ n_1, \, n_3 \}, n_2]\cap \mathbb{Z}.
$$
These connected components have the following equivalent characterizations:
\begin{align*}
\cQ^k_{\cF} & =   \{ \,\{ x_i\} \in \cQ_\cF \, :  \, j(U^*U,X^*X)=k \, \} \\
& = \{ \,\{ x_i\} \in \cQ_\cF \, :  \,  j(UU^*,XX^*)=k   \,  \} \\
& = \{  \,\{ x_i\} \in \cQ_\cF \, :  \,  \text{index}(XU^*|_\cK : \cK \to \cL)=k, \,  \cL=\overline{span} \{ x_i \} \,   \}, 
\end{align*}
where $X$ denotes the synthesis operator associated to each Parseval frame $\{ x_i \}$. Furthermore, each connected component is an orbit of the action defined in \eqref{action uxu}.
\end{pro}
\begin{rem}
In the definition of the set $\mathbb{J}$ we interpret that the interval contains all the negative numbers when $n_1=\infty$ and $n_3=\infty$, and all the positive numbers when $n_2=\infty$. 
\end{rem}
\begin{proof}
We begin with the proof of the first characterization of 
the connected components. Take a Parseval frame $\{ x_i ^0\} \in \cQ_\cF$. Denote by $\cC$ the connected component of $\{ x_i^0\}$ in $\cQ_\cF$.
 Since the synthesis operator $X_0$ of  $\{ x_i ^0\}$ is such that $U-X_0$ is Hilbert-Schmidt, then $U^*U-X_0^*X_0$ is also Hilbert-Schmidt, and in particular, by Remark \ref{comp and sum} $i)$, it follows that $(U^*U, X_0^*X_0)$ is a Fredholm pair. Set $j(U^*U,X_0^*X_0)=k_0$. We have to show that $\cC \subseteq \cQ^{k_0}_{\cF}$. Let $\cP$ be the set of all orthogonal projections on $\ell^2$. Fixed $P=U^*U \in \cP$, we consider the following Grassmann  manifold
$$
Gr(P)=\{ \, Q \in \cP  \, : \, P-Q \text{ is Hilbert-Schmidt} \, \},
$$
which is naturally endowed  with the metric $d_2(Q_0,Q_1)=\|Q_0 - Q_1 \|_2$. It was shown in \cite[Remark 3.7]{C85} that the connected component of a projection $Q_0 \in Gr(P)$ can be described as
\begin{equation}\label{carey 2}
\{  \, Q \in Gr(P) \,  : \, j(Q_0,Q)=0 \, \}=\{ \, WQ_0W^* \, : \, W \in \cU_2(\ell^2) \, \},
\end{equation}
and hence by Remark \ref{comp and sum} $ii)$, the connected components of $Gr(P)$ are given by
$$
\{ \, Q \in Gr(P) :  \, j(P,Q)=k \},
$$
where the integer $k$ varies over all the possible values of the index. For instance, note that when the kernel and range of $P$ are both infinite dimensional the index can be any integer. However, if $P$ has finite dimensional range (resp. kernel), then the index cannot take values greater than $\dim \ran(P)$ (resp. lower than $\dim \ker(P)$).
Now using that the map $\varphi: \cQ_\cF \to Gr(P)$, $\varphi(\{ x_i \})=X^*X$, is continuous,  
we obtain that $\varphi(\cC)$ is connected. This implies that the index does not change in $\cC$, and hence $\cC \subseteq \cQ^{k_0}_{\cF}$.

To prove the reversed inclusion, we will see that $\cQ^{k_0}_\cF$ is connected. Take two Parseval frames $\{ x_i\}, \{ y_i\} \in \cQ^{k_0}_{\cF}$. Then  their synthesis operators satisfy $j(U^*U, X^*X)=j(U^*U,Y^*Y)=k_0$. Using Remark \ref{comp and sum}, it follows that 
$j(X^*X,Y^*Y)=j(X^*X,U^*U) + j(U^*U,Y^*Y)=-k_0+k_0=0$. According to Lemma \ref{stiefel man},   there are two unitaries $W \in \cU_2(\ell^2)$ and $V \in \cU_2(\cH)$ such that $X=VYW^*$.
Now recall that the Lie algebra of the group $\cU_2(\cH)$ is given by
$$
\mathfrak{u}_2(\cH)=\{ \, A \in \cB_2(\cH) \, : \, A=-A^* \, \},
$$
 and the exponential map  $\exp:\mathfrak{u}_2(\cH) \to \cU_2(\cH)$, $\exp(A)=e^A$ is surjective (see e.g. \cite[Remark 3.1]{AL08}). Then there exist $A=-A^*$ and $B=-B^*$  Hilbert-Schmidt operators such that $V=e^A$ and $W=e^B$.  Thus, there is a continuous path  joining $X$ and $Y$  given by
$
\gamma(t)= e^{tA} X e^{-tB} $, $0\leq t \leq 1
$. Note that $\gamma(t) - X$ is Hilbert-Schmidt and $j(X^*X,\gamma(t)^*\gamma(t))=0$ by the characterization in \eqref{carey 2}.
Then $\{ \gamma(t)e_i \}$ is a continuous path of Parseval frames joining $\{ x_i\}$ and $\{ y_i\}$ in $\cQ_\cF^{k_0}$.   Hence $\cQ_\cF^{k_0}$ is connected. 

The second description of the connected components is a direct consequence of the first and the following fact: if two partial isometries $X,Y$ are such that $X-Y$ is Hilbert-Schmidt, then $j(X^*X,Y^*Y)=0$ if and only if $j(XX^*,YY^*)=0$. This  in turn is an immediate consequence of Lemma \ref{stiefel man}. 

The third characterization of the connected components follows by noting that $\cK=\ran(U)$, $\cL=\ran(X)$, and  that
 $X^*XU^*U|_{\ran(U^*)}:\ran(U^*)\to \ran(X^*)$ is Fredholm if and only if $XU^*|_{\ran(U)}: \ran(U)\to \ran(X)$ is. Moreover, we clearly have $j(U^*U,X^*X)=\text{index}(XU^*|_{\ran(U)}: \ran(U)\to \ran(X))$.

\medskip

In order to determine the index set $\mathbb{J}$ labeling the connected components, recall that from the first and second characterizations of $\cQ_\cF^k$, any $k \in \mathbb{J}$ must satisfy
\begin{align*}
k & =\dim(\ker(Q)\cap \ran(P)) - \dim(\ker(P)\cap \ran(Q))\\
& = \dim(\cL^\perp \cap \cK) - \dim(\cL \cap \cK^\perp),
\end{align*}
where $P=U^*U$, $Q=X^*X$ and $X$  is the synthesis operator of a Parseval frame $\{ x_i \} \in \cQ_\cF^k$ with $\ran(X)=\cL$. Clearly, this implies the estimates
$-\min\{ n_1 , \, n_3\} \leq k \leq n_2$. We now show that there exist  Parseval frames for which the index takes all these values. 
From Lemma \ref{spectral}, there are orthonormal bases $\{ g_i\}$ and $\{ Ug_i\}$ of $\ran(P)$ and $\ran(F)$, respectively. Next take  $\{ g_i '\}$ and $\{ h_i \}$ orthonormal bases of $\ker(P)$ and $\ran(F)^\perp$, respectively. 
In the case where $n_1\leq n_3$, we can define for $-n_1 \leq k < 0$, the following partial isometry:
$Xg_i=Ug_i$ for all $i\geq 1$, $Xg_i'=h_i$ for $1\leq i \leq -k$, and $Xg_i'=0$ for $i>-k$. When $k=0$, we just set $X=U$, and when $0 <k \leq n_2$, we change the definition: take $i_1, \ldots ,i_k \in [1,n_2]$, put $Xg_i=Ug_i$ for $i \neq i_1, \ldots, i_k$, $Xg_{i_j}=0$ for $j=1, \ldots ,k$, and  $Xg_i'=0$ for all $i\geq 1$.  
For any of these definitions of $X$, the operator $U-X$ has finite rank and $j(P,Q)=k$. Thus the associated Parseval frame $\{ x_i\}$ belongs to $\cQ_\cF^k$. The case where $n_1\geq n_3$ follows similarly. 
Notice that from this description of $\mathbb{J}$, and since the underlying Hilbert space $\cH$ is infinite dimensional, then one of the following conditions holds true: $\min\{n_1, \, n_3 \}=\infty$ or $n_2=\infty$. Thus, there are always infinitely many connected components of $\cQ_\cF$.

\medskip

Finally, we prove that the connected components are orbits of the  action defined in \eqref{action uxu}. Using again that the group $\cU_2(\cH)$ is path connected, it follows that the orbit
$\cO(\{ x_i\})=\{ \, \{ VXWe_i \} \, : \,V \in \cU_2(\cH), \, W \in \cU_2(\ell^2) \, \}$ is connected, and thus $\cO(\{x_i \})\subseteq \cC$, where $\cC$ is the connected component of $\{ x_i\}$. Conversely, take $\{ y_i \} \in \cC$. Then its synthesis operator is such that $Y-U$ is Hilbert-Schmidt and  $j(U^*U,Y^*Y)=j(U^*U,X^*X)$ for some integer $k \in \mathbb{J}$. From Remark \ref{comp and sum}, it follows that $Y-X$ is Hilbert-Schmidt and $j(X^*X,Y^*Y)=0$. Then there are two unitaries   $V \in \cU_2(\cH)$ and  $W \in \cU_2(\ell^2)$ satisfying 
$Y=VXW^*$ by Lemma \ref{stiefel man}. This completes the proof. 
\end{proof}

\begin{rem} 
The action is locally transitive. Indeed, the following fact was proved (\cite[Lemma 3.1]{Ch10}): If $X,Y$ are two partial isometries such that  $\| X-Y\|_2<1$, then there exist two unitaries $V \in \cU_2(\cH)$, $W \in \cU_2(\ell^2)$ such that $Y=VXW^*$. This means that Parseval frames which have $\ell^2$-distance less than $1$ belong to the same connected component. In addition, if $k_0, k_1 \in \mathbb{J}$, $k_0\neq k_1$, then the distance between the corresponding connected components satisfies
$$
d(\cQ_\cF^{k_0},\cQ_\cF^{k_1})=\inf \left\{  \, \left(\sum_{i=1}^\infty  \| x_i - y_i\|^2 \right)^{1/2} \, : \, \{ x_i \} \in \cQ_\cF^{k_0}, \, \{ y_i\} \in \cQ_\cF^{k_1} \, \right\} \geq 1.
$$
\end{rem}


\section{Symmetric approximation in connected components}\label{symm connected comp}


Our aim in this section is to study symmetric approximation by Parseval frames in the connected components of $\cQ_\cF$ described in the previous section.   Actually, this is a particular case of symmetric approximation by a family of Parseval frames as we now define.



\begin{fed}\label{general symm app}
Let $\{ f_i\}$ be a frame in a subspace $\cK$ of a Hilbert space $\cH$. Let $\cQ$ be a nonempty family of Parseval frames in subspaces of  $\cH$.
A Parseval frame $\{ y_i \} \in \cQ$  is said to be a \textit{symmetric approximation of $\{ f_i\}$ in $\cQ$} if the inequality
$$
\sum_{i=1}^\infty \| f_i - y_i\|^2 \leq \sum_{i=1}^\infty \| f_i - x_i\|^2 ,
$$
is valid for all Parseval frames $\{ x_i\} \in \cQ$ and the sum at the left side of this inequality is finite.
\end{fed}

\begin{rem}
The last assumption of the definition avoids the cases where both sides of the inequality are infinite. 
However, the right side might be infinite. 
On the other hand, we observe that this definition of symmetric approximation by a family of Parseval frames  was previously given in \cite{AC16} for the finite dimensional setting. In addition, note that  if $\cQ_\cF^w$ denotes the family of all Parseval frames weakly similar to $\cF=\{ f_i \}$, then $\{ y_i\}$ is a symmetric approximation of $\{ f_i\}$ in $\cQ_\cF^w$ when $\{ y_i\}$ is a symmetric approximation in the sense of \cite{FPT02}.
\end{rem}

\bigskip








In order to study symmetric approximation in the connected components $\cQ_\cF^k$, $k \in \mathbb{J}$, we will need three supporting results. The first one is about the critical points of the map $
X \mapsto \| F-X\|_2 ^2
$, where $F$ is the synthesis operator of the frame $\{ f_i\}$ and $X$ is any synthesis operator (partial isometry) of a quadratically close Parseval frame. This result was proved in \cite[Thm. 4.1]{M89} for a positive operator $F$, and was generalized in a remark after its proof in that article to an arbitrary operator $F$ as we state below. 

\begin{lem}\label{maher}
Suppose that $Y$ is a local extremum of the map
$
X \mapsto \| F-X\|_2 ^2
$, where $X$ varies over those partial isometries such that $F-X$ is Hilbert-Schmidt. Then $Y^*F=F^*Y$.
\end{lem}
Taking the adjoint in the definition of the above map, one also has $YF^*=FY^*$. 
By a basic result in matrix theory that goes back to \cite[Thm. II]{EY36}, the relations  $Y^*F=F^*Y$ and  $YF^*=FY^*$ between two matrices $F$, $X$ are equivalent to the existence of a  simultaneous singular value decomposition. Below we give a version  of the simultaneous singular value decomposition for the type of operators we need. Its proof is adapted from the matrix case (see \cite[Thm. 2.7.12]{MBM10}). It works in our case due the expression of the operator $F$ in Lemma \ref{spectral} and the fact that $F-Y$ is compact.

\begin{lem}\label{sim svd}
Let $F$ be an operator with closed range. Let $Y$ be a partial isometry such that $F-Y$ is a compact operator and   $YF^*=FY^*$ and $Y^*F=F^*Y$. Then there exist orthonormal bases $\{ v_i \}$ of $\ell^2$ and $\{ h_i\}$ of $\cH$, and two sequences $\{ r_i \}_1 ^\alpha$ and $\{ t_i \}_1^\beta$ such that
$$
F=\sum_{i=1}^\alpha s_i  h_{r_i} \otimes v_{r_i} \, , \, \, \, \, 1 \leq \alpha \leq \infty,
$$
and 
$$ 
Y=\sum_{i=1}^\beta \epsilon_i  h_{t_i} \otimes v_{t_i} \, , \, \, \, \, 1 \leq \beta \leq \infty,
$$
where $\{ s_i \}_1^\alpha$ are the nonzero singular values of $F$  and $\epsilon_i \in \{ -1,  1 \, \}$.
\end{lem}
\begin{proof}
 Recall that Lemma \ref{spectral} gives $F=\sum_{i=1}^\alpha s_i(Ug_i)\otimes g_i$, where $\{ g_i\}_1^\alpha$ and $\{ Ug_i \}_1^\alpha$ are  orthonormal bases of $\ker(F)^\perp$ and $\ran(F)$, respectively, and $\{ s_i\}_1^\alpha$ are the nonzero singular values of $F$.   We can complete them to orthonormal bases of $\ell^2$ and $\cH$, which induce two unitary operators $Q \in \cU(\ell^2)$ and $P \in \cU(\cH)$. In the coordinates given by these bases we write
$$
F=P\begin{pmatrix} \Delta & 0 \\ 0 & 0 \end{pmatrix}Q^*,
$$
where   $\Delta=diag(s_1, s_2 , \ldots)$ is an invertible diagonal matrix.
The compact operator defined by $A=F-Y$ also satisfies $AF^*=FA^*$ and $A^*F=F^*A$. Using the block decomposition with respect to
$\ell^2=Q^*(\ker(F))^\perp \oplus Q^*(\ker(F))$ and $\cH=P^*(\ran(F))\oplus P^*(\ran(F))^\perp$, we write in the same coordinates as above
$$
A=P	\begin{pmatrix} B & C \\ D & E \end{pmatrix} Q^*.
$$
Next note that $FA^*$ is self-adjoint, so $P^*FA^*P$ is also self-adjoint, and
$$
P^*FA^*P=(P^*FQ)(Q^*A^*P)=\begin{pmatrix} \Delta B^* & \Delta D^* \\  0 & 0 \end{pmatrix}.
$$
This implies that $D=0$ and $\Delta B^*=B \Delta$. Applying a similar argument to the self-adjoint matrix $Q^*F^*AQ$, 
one sees that $C=0$ and $\Delta B=B^*\Delta$. Therefore, $\Delta^2 B =B \Delta ^2$, which gives $\Delta B =B \Delta $ and $B=B^*$.
Since $\Delta$ is diagonal, there are a block unitary matrix $R$ and a real diagonal matrix $\Lambda$ such that 
$R^*\Delta R=\Delta$ and $R^*B R=\Lambda$. The size of the blocks of $R$ are given by the multiplicities of the nonzero singular values $\{ s_i \}_1^\alpha$.

Next we use that $E$ is also a compact matrix, so it admits a singular value decomposition, which in matrix form means that $E=X\Gamma Z^*$ for two unitaries $X$ and $Z$, and a diagonal matrix $\Gamma$ with non negative entries.  
Set 
$$
W=P\begin{pmatrix}  R & 0 \\ 0 & X \end{pmatrix}, \, \, \, \, V=Q \begin{pmatrix}  R & 0 \\ 0 & Z \end{pmatrix}.
$$
It follows that 
$$
W^*FV=\begin{pmatrix}   \Delta & 0 \\ 0 & 0 \end{pmatrix}, \, \, \, \, 
W^*AV=\begin{pmatrix}   \Lambda & 0 \\ 0 & \Gamma \end{pmatrix}. \,  \, \, \, 
$$
Therefore 
$$
W^*YV=\begin{pmatrix}     \Delta - \Lambda  & 0 \\ 0 & -\Gamma
\end{pmatrix}
$$
is a real diagonal matrix. Since $Y$ is a partial isometry, $W^*YV$ is also a partial isometry, and thus $V^*Y^*YV=(W^*YV)^*(W^*YV)$ is a projection. Hence the only possible values of the diagonal entries of $W^*YV$ are $-1,0,1$. Then the desired orthonormal bases  $\{ v_i \}$ and $\{ h_i\}$ are defined  by the unitaries $V$ and $W$. The sequences $\{ r_i\}_1^\alpha$ and $\{ t_i\}_1^\beta$ can be  chosen to satisfy  that  $\{ v_{r_i} \}_1 ^\alpha$  and $\{ h_{r_i}\}_1^\alpha$ are orthonormal bases of $\ker(F)^\perp$ and $\ran(F)$, and 
 $\{ v_{t_i} \}_1 ^\beta$  and $\{ h_{t_i}\}_1^\beta$ are orthonormal bases of $\ker(Y)^\perp$ and $\ran(Y)$. 
\end{proof}


Using the previous lemma, we will  reduce problems of existence and uniqueness of symmetric approximation of frames to similar problems for diagonal operators. This can be thought as the  optimization problem involving sequences that we discuss below.

Fix $\{ a_i \}$ a sequence of non negative numbers such that 
\begin{equation}\label{suc hip}
\sum_{i \, : \, a_i \neq 0 }(a_i-1)^2<\infty.
\end{equation}
Consider the  set
$$
\cB:=\left\{ \, \{ b_i \} \, : \, b_i \in \{ -1 , \, 0 , \, 1 \, \}, \, \, \, \sum_{i=1}^\infty (a_i - b_i)^2 < \infty \, \right\}.
$$
 For every sequence $\{ b_i\} \in \cB$, it is straightforward to see that  the following version of the index  is well defined
$$
j(\{ a_i\}, \{ b_i \}):=\# \{ \, i \, : \, b_i=0, \, a_i>0 \,\}-\# \{ \, i \, : \, b_i\neq 0, \, a_i=0 \,\}.
$$
Actually, this is the index of the pair of  projections onto
the orthogonal complements of the kernels of the diagonal operators defined by $\{ a_i\}$ and $\{ b_i\}$. The values of the index give a decomposition of $\cB$ in the following sets: 
$$
\cB^k:=\{ \, \{ b_i\} \in \cB \, : \,j(\{ a_i\}, \{ b_i \})=k  \, \}, 
$$
where the integers $k$ can  take all the values 
\begin{equation}\label{bounds index}
-\#\{ \, i \, : a_i = 0 \, \} \leq k \leq \# \{ \, i \, : a_i >0 \, \}.
\end{equation}
Here the lower and upper bounds can be $-\infty$ and $\infty$.
Notice that the value $k=0$ is always achieved   by choosing $b_i=1$ if $a_i >0$ and $b_i=0$ if $a_i=0$, and using the assumption in \eqref{suc hip}. 

\begin{lem}\label{diagonal}
With the notation above, consider the function
$$
f: \cB \to \R, \, \, \, \, f(\{ b_i\})=\sum_{i=1}^\infty (a_i-b_i)^2 \,.
$$
For any integer $k$ satisfying \eqref{bounds index},  the minimizers of the function $f$ restricted to the set $\cB^k$ are given as follows:
\begin{itemize}
\item[i)] If $k=0$, then there is a unique minimizer given by
$$
b_i = \begin{cases} 1 & \mbox{if }\, \, a_i \neq 0, \\ 0 & \mbox{if } \, \, a_i=0. \end{cases}
$$
\item[ii)] If  $k<0$, then the set $
\mathbb{A}^k:= \{ \, i \, : \, a_i=0 \, \}$ satisfies $\#\mathbb{A}^k\geq -k$. There are $2^{-k}\binom{\#\mathbb{A}^k}{-k}$  minimizers if $\#\mathbb{A}^k< \infty$, or infinitely many minimizers if $\#\mathbb{A}^k= \infty$, which   can be described as follows.
For $i_1, \ldots, i_{-k} \in \mathbb{A}^k$, 
\begin{equation}\label{minimizers k<0}
b_i  = \begin{cases} 1 & \mbox{ if }\, \, a_i \neq 0, \\ 
\pm 1 & \mbox{ if }  \, \, i=i_1, \ldots, i_{-k}, \\
0 & \mbox{ if } \, \,  a_i=0 \mbox{ and } i \neq i_1, \ldots, i_{-k}. \end{cases}
\end{equation}
\item[iii)] If $k>0$, then there exists a minimizer  if and only if there are $k$ nonzero lowest numbers $a_{n_1}\leq  \ldots \leq a_{n_k}$ in the sequence $\{ a_i \}$. In this case, set 
$$
\mathbb{A}^k:=\{ \,\{n_1, \ldots, n_k\}\, : \, 0<a_{n_1}\leq \ldots \leq a_{n_k}\leq a_j\, , \,  \,  \forall \, j\neq n_1, \ldots, n_k, \, a_j \neq 0  \,\}.
$$ 
Then all the minimizers can be described as follows: for $\{n_1, \ldots, n_k\} \in \mathbb{A}^k$,
$$
b_i  = \begin{cases} 1 & \mbox{ if }\, \, a_i \neq 0, \, i \neq n_1 , \ldots , n_k, \\ 
 0 & \mbox{ if }  \, \, a_i=0 \mbox{ or } i=n_1, \ldots, n_{k}. \end{cases}
$$
Thus,  there is a unique minimizer if and only if $\#\mathbb{A}^k=1$.
\end{itemize}
\end{lem}
\begin{proof}
\noi $i)$  The condition imposed by the index and the inequality $(a-1)^2<a^2+1$ for $a>0$, implies that any candidate to be a minimizer $\{ b_i\}$ must satisfy 
$$
\# \{ \, i \, : \, b_i=0, \, a_i>0 \,\}=\# \{ \, i \, : \, b_i\neq 0, \, a_i=0 \,\}=0.
$$
This means that $b_i=\pm 1$ if $a_i>0$, and $b_i=0$ if $a_i=0$.  Then the inequality $(a-1)^2<(a+1)^2$ for $a>0$ can be used to deduce that there is a unique minimizer of the desired form.  

\medskip

\noi $ii)$ Note that $\# \mathbb{A}^k\geq \# \{ \, i \, : \, b_i\neq 0, \,  a_i=0\, \}\geq -k$ because $\cB^k\neq \emptyset$.  The proof is similar to the previous item, we only remark that for the indices $i_1, \ldots , i_{-k}$, we can have $b_{i_j}=1$ or $b_{i_j}=-1$. Hence  there are $2^{-k}\binom{\#\mathbb{A}^k}{-k}$  minimizers if $\#\mathbb{A}^k< \infty$, or infinitely many minimizers if $\#\mathbb{A}^k= \infty$.

\medskip

\noi $iii)$ Suppose that there are $k$ values of the sequence $\{ a_ i\}$ such that $0<a_{n_1}\leq \ldots \leq a_{n_k}\leq a_j$ for all $j \neq n_1, \ldots, n_k$, $a_j\neq 0$. Again from the condition given by the index and the inequality $(a-1)^2<a^2+1$ for $a>0$, any candidate to be a minimizer $\{ b_i\}$ have to satisfy
\begin{equation}\label{algunos zeros}
\# \{ \, i \, : \, b_i=0, \, a_i>0 \,\}=k , \, \, \, \, \# \{ \, i \, : \, b_i\neq 0, \, a_i=0 \,\}=0.
\end{equation}
Then, using the inequality $(a-1)^2+ b^2 < a^2 + (b-1)^2$ for $b < a$, we conclude that if we place $k$ zeros on the indices $ n_1, \ldots, n_k$ of a sequence $\{ b_i \}$, then we obtain a minimizer. The set $\mathbb{A}^k$ is defined to take into account the possibly repetition of the $k$ lowest values of $\{ a_i\}$ which leads to other minimizers. Indeed, if $\{ n_1 , \ldots , n_k\}, \, \{n_1' , \ldots , n_k' \} \in \mathbb{A}^k$, then $a_{n_1}=a_{n_1'}, \ldots, a_{n_k}=a_{n_k'}$. This implies that  $f(\{ b_i\})=f(\{ b_i'\})$, where $\{ b_i\}$ and $\{ b_i'\}$ are the minimizing sequences in $\cB^k$ associated to $\{ n_1 , \ldots , n_k\}$ and $\{n_1' , \ldots , n_k' \}$, respectively.  On the other hand, it is now clear that there is unique minimizer when  $\#\mathbb{A}^k=1$.

In order to prove the converse, assume that do not exist $k$ lowest values of $\{ a_i\}$. 
As before, if we suppose that there is minimizer $\{ b_i \} \in \cB^k$, then it must satisfy \eqref{algunos zeros}. Thus, $b_{n_1}=\ldots =b_{n_k}=0$ for some integers $n_1, \ldots , n_k$. By hypothesis there exists $n_0 \neq n_1, \ldots, n_k$ such that $0<a_{n_0}<a_{n_j}$ for some $j=1, \ldots, k$. We define another sequence $\{ b'_i \} \in \cB^k$ as follows: $b'_i=b_i$ if $i \neq n_0 , n_j$, $b'_{n_0}=0$, and $b_{n_j}=1$. Since $(a_{n_j}-1)^2 + a_{n_0}<a_{n_j}^2 + (a_{n_0}-1)^2$, we have $f(\{ b_i'\})<f(\{ b_i\})$. This is a contradiction, which shows that the function $f$ does not have minimizers in this case. 
\end{proof}



Our first result on  symmetric approximation now follows. It gives the existence and uniqueness of symmetric approximation in the connected component $\cQ_\cF^0$. In the statement we assume the existence of quadratically close Parseval frames, which  is  a necessary condition to seek for symmetric approximations.

\begin{teo}\label{canonico}
Let $\cF=\{ f_i\}$ be a frame in a subspace $\cK\subseteq \cH$, and let $\{ u_i\}$ be its canonical Parseval frame. Suppose that there exists a Parseval frame quadratically close to $\{ f_i\}$. Then the estimate
$$
\sum_{i=1}^\infty \| f_i - u_i \|^2 \leq \sum_{i=1}^\infty \| f_i - x_i \|^2
$$
is valid for all Parseval frames $\{ x_i\} \in \cQ_\cF^0$. Furthermore, equality holds if and only if $\{ x_i \}=\{ u_i \}$.
\end{teo}
\begin{proof}
The proof  of the existence of symmetric approximations is similar to Theorem \ref{cuadratically close}, but  we do not have to deal with the Fredholm unitary group and the Ky-Fan norms now. We begin by noting that by Theorem \ref{cuadratically close}, there exists a Parseval frame quadratically close to $\{ f_i\}$ if and only if $\sum_{i=1}^\infty \| f_i -u_i\|^2<\infty$. This gives that $\cQ_\cF^0 \neq \emptyset$ and also the last condition in  Definition \ref{general symm app} of symmetric approximation is verified.   

 As usual, we write $F=U|F|$ for the polar decomposition. Now take an arbitrary Parseval frame $\{ x_i\} \in \cQ_\cF^0$. According to Proposition \ref{compo con}, this means that its synthesis operator $X$ is a partial isometry such that $X-U$ is Hilbert-Schmidt and $j(P,Q)=0$, where $Q=X^*X$ and $P=U^*U$ is the orthogonal projection onto $\ker(F)^\perp$. It follows that $P-Q$ is Hilbert-Schmidt.
As we stated in Remark \ref{holds for p2}, Theorem \ref{carey} holds if we replace  compact operators by  Hilbert-Schmidt operators and the Fredholm unitary group by the group $\cU_2(\cH)$. Hence
 there exists $L \in \cU_2(\ell^2)$ such that $Q=LPL^*$.

Applying Lemma \ref{spectral}, we obtain  an orthonormal basis $\{ g_i \}_1^\alpha$ of $\ker(F)^\perp$  satisfying
$
F=\sum_{i=1}^\alpha s_i (Ug_i) \otimes g_i
$, 
where $\{ s_i\}_1^\alpha$ are the nonzero singular values of $F$.  In what follows, we assume that $\alpha=\infty$ (the case $\alpha<\infty$ is easier since we do not need an approximation argument). Given $n \geq 1$, set
$$
P_n:= \sum_{i=1}^n  g_i \otimes g_i  \, ; \, \, \, \, \, \, Q_n:=LP_nL^*.
$$
Note that $\rank(P_n)=n$, $P_n\leq P$ and $Q_n\leq Q$.
Since the singular values satisfies
$s((F-U)P_n)=(|s_1-1|, \ldots, |s_n -1|)^{\downarrow}$,
$s(FP_n)=(s_1, \ldots , s_n)^{\downarrow}$ and $s(XQ_n)=(1, \ldots , 1)$,
it follows that
$$
s((F-U)P_n)=|s(FP_n) - s(XQ_n)|^{\downarrow}, \, \, \, \, n \geq 1.
$$
Then, by Theorem \ref{mirsky}, for all $m\geq 1$ we get 
\begin{align*}
\sum_{i=1}^m s_i((F-U)P_n) & = \sum_{i=1}^m |s(FP_n)-s(XQ_n)|_i ^{\downarrow} \\
& \leq \sum_{i=1}^m s_i(FP_n-XQ_n).
\end{align*}
This implies that 
$\|(F-U)P_n\|_2\leq \|FP_n - XQ_n\|_2$ by the  dominance property  \cite[p. 82]{GK60}.  Noting that  $P_n \to P$, $Q_n \to Q$ strongly, and that the operators $F-U$, $F-X$ and $L-I$ are Hilbert-Schmidt, we have $\|(F-U)P_n - (F-U)\|_2 \to 0$ and $\|FP_n - XQ_n - (F-X)\|_2 \to 0$
 (see \cite[Thm. 6.3]{GK60}). This gives 
$\| F-U\|_2\leq \|F-X\|_2$. Hence $\{ u_i \} = \{ Ue_i \}$ is a symmetric approximation of $\{ f_i\}$ in $\cQ_\cF^0$.  

\medskip

Now we prove the uniqueness. Suppose that $\{ y_i \}$ is another symmetric approximation  of the frame $\{ f_i\}$ in $\cQ_\cF^0$. This means that its synthesis
operator $Y$  satisfies $\|F-Y\|_2 \leq \|F-X\|_2$, for all  partial isometries $X$ which are  synthesis operators of Parseval frames $\{ x_i\} \in \cQ_\cF^0$. Therefore $Y$ is a local minimum of the differentiable map $X\mapsto \| F-X\|_2^2$, and by Lemma \ref{maher}, it follows that   $Y^*F=F^*Y$ and $YF^*=FY^*$. According to Lemma \ref{sim svd}, there exist orthonormal bases $\{ v_i \}$ of $\ell^2$ and $\{ h_i\}$ of $\cH$, and two sequences $\{ r_i \}_1 ^\alpha$ and $\{ t_i \}_1^\beta$ such that
$$
F=\sum_{i=1}^\alpha s_i  h_{r_i} \otimes v_{r_i} \, , \, \, \, \, \, \, \, Y=\sum_{i=1}^\beta \epsilon_i  h_{t_i} \otimes v_{t_i}\, ,
$$
where $\{ s_i \}_1^\alpha$ are the nonzero singular values of $F$  and $\epsilon_i \in \{ -1,  1 \, \}$. Using the notation of Lemma \ref{diagonal}, set
$$
a_i= \begin{cases} s_i  , \, \, \, \mbox{if } i = r_j \mbox{ for some } 1\leq j \leq \alpha, \\ 0, \, \, \, \mbox{ otherwise}; \end{cases}
\, \, \, \, \, \, \, \, \, \,
\tilde{b}_i= \begin{cases} \epsilon_i ,\, \, \, \mbox{if } i = t_j \mbox{ for some } 1\leq j \leq \beta, \\ 0, \, \, \, \mbox{ otherwise}. \end{cases}
$$
For each sequence $\{ b_i\} \in \cB^0$, we define a  partial isometry  $X=\sum_{i=1}^\infty b_i h_i \otimes v_i \in \cQ_\cF^0$. Then, 
\begin{align*}
  f(\{ b_i\}) & = \sum_{i=1}^\infty (a_i - b_i)^2 =\sum_{i=1}^\infty \| (F-X)v_i \|^2 
	= \| F-X \|_2^2 \\
& \geq   \|F-Y\|_2^2 
 = \sum_{i=1}^\infty \| (F-Y)v_i \|^2 = \sum_{i=1}^\infty (a_i - \tilde{b}_i)^2 =f(\{ \tilde{b}_i\}). 
\end{align*}
By Lemma \ref{diagonal} $i)$, we have that $\{ \tilde{b}_i\}$ is the unique minimizer of the function $f$ restricted to $\cB^0$. Thus, $\tilde{b}_i=1$ if $a_i >0$, and $\tilde{b}_i=0$ if $a_i=0$.
Hence,
$$
Y =\sum_{i=1}^\infty \tilde{b}_i h_i \otimes v_i 
 =\sum_{i=1}^\alpha h_{r_i} \otimes v_{r_i}   =U,
$$
which implies that $\{ y_i\}=\{ u_i\}$. 
\end{proof}

\begin{rem}
Given an arbitrary frame $\cF=\{ f_i\}$, consider the set 
$$
\cQ_\cF^w = \left\{ \, \{ x_i\}   \, : \, \{ x_i\} \text{ Parseval frame weakly similar to } \{ f_i\}, \, \sum_{i=1}^\infty\| x_i - f_i \|^2 < \infty \,  \right\}.
$$ 
Since $\cQ_\cF^w \subseteq \cQ_\cF^0$, it follows that Theorems \ref{cuadratically close} and \ref{canonico} generalize \cite[Thm. 2.3]{FPT02}. Further, the geometric structure of $\cQ^w_\cF$ can also be studied  using previous results for partial isometries  (see  \cite[Thm. 2.3]{Ch10}). For instance, we have that $\cQ^w_\cF$
  coincides with the orbit of the following action:
$$
 V \cdot \{ x_i \}=\{ VXe_i \}, \, \, \, V \in \cU_2(\cH),
$$
where $\{ x_i\}$ is any Parseval frame belonging to $\cQ_\cF^w$ with synthesis operator $X$. As a straightforward consequence  $\cQ_\cF^w$ is connected.  
\end{rem}

\bigskip


 We now study symmetric approximation restricted to connected components of the form $\cQ_\cF^k\neq \emptyset$, $k<0$.  We first observe that, according to Proposition \ref{compo con}, $\cQ_\cF^k\neq \emptyset$ for some $k<0$ if and only if
\begin{equation}\label{nonemp}
n_1=\dim \ker(F) >0 \, \, \, \text{ and }\,  \,  n_3=\dim \cK^\perp >0.
\end{equation}
  Our next result shows that  symmetric approximations in these type of connected components always exist, and it is possible to give a complete description of them.  


\begin{teo}\label{S aprox}
Let $\cF=\{ f_i \}$ be a frame in a subspace $\cK\subseteq \cH$. Let $F$ be its synthesis operator, $F=U|F|$ the polar decomposition and $\{ u_i\}=\{ Ue_i\}$ the canonical Parseval frame. Assume that both conditions in \eqref{nonemp} hold. Then for any integer satisfying $ -\min\{ n_1, \, n_3 \} \leq  k \leq -1 $, and every partial isometry $S:\ker(F)\to \ran(F)^\perp$  such that $\rank(S)=-k$, the estimate
$$
\sum_{i=1}^\infty \| f_i - u_i-S(I-U^*U)e_i \|^2 \leq  \sum_{i=1}^\infty \| f_i - x_i \|^2
$$ 
is valid for all Parseval frames  $\{ x_i\} \in \cQ^k_\cF$. Thus, there are always infinitely many symmetric approximations of $\{ f_i\}$ in $\cQ_\cF^k$. Moreover, all the  symmetric approximations are given by $\{ u_i+S(I-U^*U)e_i \}$, where $S$ is a partial isometry satisfying the above conditions. 
\end{teo}
\begin{proof}
 Take an integer $k$ under the above assumptions, and let $\{ x_i\} \in \cQ_\cF^k$ be a Parseval frame in a subspace $\cL \subseteq \cH$. Denote by $X$ its synthesis operator, and set  $Q=X^*X$ and $P=U^*U$. Since $j(P,Q)=\dim(\ker(Q)\cap \ran(P))-\dim(\ker(P)\cap \ran(Q))=k$, it follows that $\dim(\ker(F))=\dim(\ker(P))\geq -k$. By Proposition \ref{compo con},  the second characterization of the connected components, yields $j(UU^*,XX^*)=\dim(\cK \cap \cL^\perp)-\dim(\cK^\perp \cap \cL)=k$. Thus, $\dim \cK^\perp=\dim \ran(F)^\perp\geq -k$. Hence it is always possible to construct partial isometries $S:\ker(F)\to \ran(F)^\perp$ satisfying $\rank(S)=-k$.

Now we consider the partial isometry $U_S:=U+S(I-P)$, which is the synthesis operator of the Parseval frame $\{ u_i+S(I-U^*U)e_i \}$. We express the initial projection of $U_S$ as $P'=P+P''$, where $P''=(I-P)S^*S(I-P)$. Noting that $P' -P$ is a finite rank operator, by Remark \ref{comp and sum} $ii)$ the following formula for the index is applicable: $j(P',Q)=j(P',P)+j(P,Q)=0$. According to the remark after Theorem \ref{carey},  there exists $L \in \cU_2(\ell^2)$ such that $Q=LP'L^*$. On the other hand, by Lemma \ref{spectral}, the operator $F$ may be expressed as
$$
F=\sum_{i=1}^\alpha s_i (Ug_i) \otimes g_i, \, \, \, \, 1 \leq \alpha \leq \infty.
$$ 
 We now assume that $\alpha=\infty$. The case where $\alpha<\infty$, i.e. $F$ a finite rank operator,  follows similarly and is easier (it is not necessary to approximate the projections $P'$ and $Q$ as below). For $n \geq 1$, we define the projections
$$
P_n=\sum_{i=1}^n g_i \otimes g_i\, + P''; \, \, \, \,Q_n:=L^*P_nL. 
$$
Observe that $\rank(P_n)=\rank(Q_n)=n-k$, $P_n\leq P'$, $Q_n\leq Q$, and $P_n \to P'$, $Q_n \to Q$ strongly. Since
$$
|(F-U_S)P_n|=\sum_{i=1}^n |s_i -1| g_i \otimes g_i + P'',
$$ 
it follows that 
\begin{align*}
s((F-U_S)P_n) & = (|s_1 - 1|, \ldots, |s_n-1|, \underbrace{1, \ldots, 1}_{-k})^\downarrow \\
& =| (s_1, \ldots, s_n,\underbrace{0, \ldots, 0}_{k} )^\downarrow - (\underbrace{1, \ldots, 1}_{n-k})|^\downarrow \\
& =|s(FP_n) - s(XQ_n)|^\downarrow. 
\end{align*}
 Using Theorem \ref{mirsky} and the same approximation argument of Theorem \ref{canonico}, we can conclude that 
$\{ u_i +S(I-P)e_i \}$ is a symmetric approximation of $\{ f_i \}$ in $\cQ_\cF^k$. 

It remains to show that every symmetric approximation of $\{ f_i \}$ in $\cQ_\cF^k$ ($k \in \mathbb{J}$, $k<0$) is given by that formula. Take 
$\{ y_i \}$ such a symmetric approximation. Its synthesis
operator $Y$  satisfies $\|F-Y\|_2 \leq \|F-X\|_2$, for all  partial isometries $X$ which are  synthesis operators of Parseval frames $\{ x_i\} \in \cQ_\cF^k$. Hence $Y$ is a local minimum of the differentiable map $X\mapsto \| F-X\|_2^2$.  As in the proof of the uniqueness part of Theorem \ref{canonico}, recall that Lemmas \ref{maher} and \ref{sim svd} can be used to reduce the problem to the diagonal case. By Lemma \ref{diagonal} $ii)$, we have that each minimizer $\{ b_i\}$ is defined by $-k$ numbers $i_1 , \ldots , i_{-k} \in \mathbb{A}^k$. Therefore,
$$
Y=\sum_{i=1}^\infty b_i \, h_i \otimes v_i =\sum_{i=1}^\alpha \, h_{r_i} \otimes v_{r_i} +\sum_{j=1}^{-k} \epsilon_j h_{i_j} \otimes v_{i_j}= U + \sum_{j=1}^{-k} \epsilon_j h_{i_j} \otimes v_{i_j}; \, \, \, \epsilon_j=\pm 1.
$$ 
Then, $S:=\left(\sum_{j=1}^{-k} \epsilon_j h_{i_j} \otimes v_{i_j}\right)|_{\ker(F)}: \ker(F)\to \ran(F)^\perp$ is a partial isometry of rank $-k$ such that 
$y_i=u_i+S(I-P)e_i$.
\end{proof}

\medskip

 In contrast to the previous result, symmetric approximations in $\cQ_\cF^k$ ($k>0$, $k \in \mathbb{J})$ do not always exist. 
 Below we prove  a necessary and sufficient condition for existence of symmetric approximation restricted to these connected components  in terms of the singular values of the synthesis operator.  These singular values are also useful to give a uniqueness criterion, and to the describe all the possible symmetric approximations.

\begin{teo}\label{k>0 aprox}
Let $\cF=\{ f_i \}$ be a frame in a subspace $\cK\subseteq \cH$. Let $F$ be its synthesis operator, $F=U|F|$ the polar decomposition and $\{ u_i\}=\{ Ue_i\}$ the canonical Parseval frame. Denote by $\{ s_i\}_1^\alpha$ ($1\leq \alpha \leq \infty$) the nonzero singular values of the operator $F$. Assume that there exists a quadratically close Parseval frame to $\{f_i \}$. Then for any integer  $1 \leq k \leq \alpha$ the following statements hold: 
\begin{itemize}
\item[i)] There exists a symmetric approximation of the frame $\{ f_i \}$ in $\cQ_\cF^k$  if and only if there are $k$  nonzero singular values $\{ s_{n_i}\}_1^k$   satisfying 
$s_{n_1}\leq  \ldots \leq s_{n_k}\leq s_j$,
 for all $j\neq n_1, \ldots , n_k$.
\item[ii)] Assume that the previous condition for existence holds, and define
$$
\mathbb{A}:= \{ \, j \, : \, 0<s_j < s_{n_k}  \, \}, \, \, \, \, \, \, \mathbb{B}:= \{ \, j \, : \, s_j = s_{n_k} \, \}. 
$$
Then all the symmetric approximations in $\cQ_\cF^k$ are given by  
\begin{equation}\label{k positivo}
\left\{ u_i- U(E_0 +E_1)e_i \right\},
\end{equation}
where $E_0$ is the orthogonal projection onto $\displaystyle{\oplus_{j \in \mathbb{A}}\ker(|F|-s_j I)}$ and $E_1$  is any orthogonal projection such that $\ran(E_1)\subseteq \ker(|F|-s_{n_k}I)$ and $\rank(E_1)=k-\#\mathbb{A}$. In particular, there is a unique symmetric approximation if and only if $\#\mathbb{A}+ \#\mathbb{B}=k$. Otherwise, there are infinitely many symmetric approximations.
\end{itemize}
\end{teo}
\begin{proof}
We are assuming that there exists a quadratically close Parseval frame to $\{ f_i\}$, then by Proposition \ref{compo con} we have   $\cQ_\cF^k\neq \emptyset$ for any integer $k$ satisfying $1\leq   k  \leq n_2=\alpha$, where $\alpha$ is the number of nonzero singular values of $F$.  So it makes sense to look for symmetric approximations in these type of components.


\medskip

\noi $i)$ Assume that there exist $k$  singular values of the synthesis operator $F$ such that  $0<s_{n_1} \leq \ldots \leq s_{n_k}\leq s_j$, for all $j \neq n_1 , \ldots, n_k$, $s_j \neq 0$. We will prove that the Parseval frame defined in \eqref{k positivo} is a symmetric approximation. Suppose that the above sets are given by
$$
\mathbb{A}=\{ \, n_1, \ldots, n_t \, \}, \, \, \, \, \, \, \, \mathbb{B}=\{ \, n_{t+1}, \ldots , n_k \, \} \cup \{ \, n_i' \, : \, 1\leq i \leq \gamma \, \},
$$ 
where $0\leq t <k$ and $0 \leq \gamma \leq \infty$. Here we interpret that $\mathbb{A}=\emptyset$ if $t=0$, and $\{ \, n_i' \, : \, 1\leq i \leq \gamma \, \}=\emptyset$ if $\gamma=0$. Note that 
$1$ is the unique nonzero singular value which could have infinite multiplicity (this happens when $\gamma = \infty$). 

Let $E_1$ be an orthogonal projection such that $\ran(E_1)\subseteq \ker(|F|-s_{n_k}I)$ and $\rank(E_1)=k-t$.  This projection can be written as 
$
E_1=\sum_{i=1}^{k-t} v_i \otimes v_i,
$
where $\{ v_i \}_1^{k-t}$ is an orthonormal basis of $\ran(E_1)$. An examination of the proof of Lemma \ref{spectral} combined with the fact 
$v_i \in \ker(|F|-s_{n_k}I)$ for all $i=1, \ldots , k-t$, yields that we may choose an orthonormal basis $\{ g_i\}_1^\alpha$ of $\ker(F)^\perp$ such that 
$$
F=\sum_{i=1}^\alpha s_i (Ug_i)\otimes g_i
$$ 
and $g_{n_{t+1}}=v_1, \ldots , g_{n_k}=v_{k-t}$. Now note that the projection $E_0$ defined in the statement is given by 
$
E_0=\sum_{i=1}^t g_{n_i}\otimes g_{n_i}.
$

Pick a Parseval frame $\{ x_i\}$ in $\cQ_\cF^k$, $k>0$, with synthesis operator $X$. Set  $Q=X^*X$, $P=U^*U$ and   $\{ m_1 , m_2 , \ldots \}=\N\setminus \{ n_1, \ldots, \, n_k \, \}$. We assume that $\alpha=\infty$ (the case $\alpha < \infty$ follows similarly). We define the following projections:
$$
P':=\sum_{i=1}^\infty g_{m_i} \otimes g_{m_i} \, ; \, \, \,  \, \, \, \, \, P'':=\sum_{i=1}^k g_{n_i}\otimes g_{n_i}=E_0 + E_1 \, .
$$
Thus, $P=P'+P''$. Again noting that $P-P'$ is a finite rank operator, $j(P',Q)=j(P',P)+ j(P,Q)=-k+k=0$. From Remark \ref{holds for p2},  there is a unitary 
$L \in \cU_2(\ell^2)$ such that $Q=LP'L^*$. Set
$$
P_n:= \sum_{i=1}^n g_{m_i} \otimes g_{m_i}  \, ; \, \, \,  \, \, \, \, \, Q_n:=LP_nL^*.
$$
We have that $\rank(P_n)=\rank(Q_n)=n$, $P_n\leq P'$, $Q_n\leq Q$, and $P_n \to P'$, $Q_n \to Q$ strongly.
Note that the partial isometry
$$
U_k=\sum_{i=1}^\infty (Ug_{m_i})\otimes g_{m_i}
$$ 
is the synthesis operator of the Parseval frame defined in \eqref{k positivo}.    
The assumption on the $k$ lowest singular values is used now to write in non increasing order  the following vector of singular values:
$$
s(F(P_n+P''))=((s_{m_1}, \ldots, s_{m_n})^\downarrow, s_{n_k}, \ldots, s_{n_1}).
$$
Note that
$$
s(XQ_n)=(\underbrace{1, \ldots, 1}_{n}, \underbrace{0, \ldots, 0}_{k}),
$$
which gives
\begin{align*}
s((F-U_k)(P_n+P'')) & = (|s_{m_1}-1|, \ldots, |s_{m_n} -1|, s_{n_k}, \ldots , s_{n_1} )^\downarrow \\
&=|s(F(P_n+P''))-s(XQ_n)|^\downarrow \, .
\end{align*}
Now the result follows by applying Theorem \ref{mirsky} and the usual approximation argument. 

\medskip

For the converse, we observe that the proof can be reduced to the diagonal case as in the uniqueness part of Theorem \ref{canonico}, and then we can use  Lemma \ref{diagonal} $iii)$ to derive the result. 

\medskip
\noi $ii)$ We have proved that the Parseval frame defined in \eqref{k positivo} is a symmetric approximation of $\{ f_i \}$ in $\cQ_\cF^k$.
The fact that every such symmetric approximation has this form follows again  by an application of  Lemmas \ref{maher}, \ref{sim svd} and \ref{diagonal} $iii)$ as in the uniqueness part of Theorem \ref{canonico}. Note that if there is a unique symmetric approximation, then there is only one orthogonal projection $E_1$ such that $\rank(E_1)=k - \#\mathbb{A}$ and $\ran(E_1)\subseteq \ker(|F|-s_{n_k}I)$. 
It is not difficult to see that 
$E_1$ must be the orthogonal projection onto $\ker(|F|-s_{n_k}I)$. Hence we obtain
$\#\mathbb{B}=\dim \ker(|F|-s_{n_k}I)=\rank(E_1)=k-\#\mathbb{A}$. 
Conversely, suppose that $\#\mathbb{A} + \#\mathbb{B}=k$, then the projection $E_1$ onto  $\ker(|F|-s_{n_k}I)$ is the unique orthogonal projection satisfying the conditions $\rank(E_1)=k - \#\mathbb{A}$ and $\ran(E_1)\subseteq \ker(|F|-s_{n_k}I)$. Since every symmetric approximation has the form \eqref{k positivo}, we get that there is a unique symmetric approximation. 

In the case where symmetric approximation is  not unique, it follows that $k-\#\mathbb{A}< \#\mathbb{B}=\dim \ker(|F|-s_{n_k}I)$. Therefore there are infinitely many orthogonal projections $E_1$ satisfying $\rank(E_1)=k - \#\mathbb{A}$ and $\ran(E_1)\subseteq \ker(|F|-s_{n_k}I)$, and consequently, there are infinitely many symmetric approximations. 
\end{proof}

\section{Global symmetric approximation}\label{section global}

Let $\cF=\{ f_i \}$ be a frame in a subspace $\cK\subseteq \cH$.   In this section we discuss existence and uniqueness of symmetric approximations of $\{ f_i \}$ in the set $\cQ_\cF$  of all Parseval frames quadratically close to $\{ f_i\}$. 

\begin{rem}\label{distance}
From the proofs of Theorems \ref{canonico}, \ref{S aprox} and \ref{k>0 aprox}, when there exists a symmetric approximation in $\cQ_\cF^k$, the $\ell^2$-distance to these connected components can be expressed in terms of the nonzero singular values $\{ s_i\}_1^\alpha$ of the synthesis operator $F$:
\begin{itemize}
\item $k=0$: $d^2(\{ f_i\}, \cQ_\cF^0)=\sum_{i=1}^\alpha (s_i-1)^2$.
\item $k<0$: $d^2(\{ f_i\}, \cQ_\cF^k)=\sum_{i=1}^\alpha (s_i-1)^2 +  1 + \ldots + 1$, where the number of $1$'s is $-k$.
\item $k>0$: $d^2(\{ f_i\}, \cQ_\cF^k)=\sum_{i \neq n_1 , \ldots, \, n_k } (s_i-1)^2 + \sum_{i=1}^k s_{n_i}^2$.
\end{itemize}
\end{rem}

Next we need to recall a result of best approximation  by partial isometries in the finite dimensional setting. 

\begin{rem}\label{finite dim case}
Let $A$ be an $m \times n$ matrix with nonzero singular values 
$$
s_1 \geq \ldots \geq s_p > 1/2 \geq s_{p+1} \geq \ldots \geq s_q, \, \, \,  \, q\leq \min\{ m, \, n \}. 
$$ 
Suppose that $A=V\Sigma W^*$ is a singular  value decomposition of $A$, i.e. $V, W$ are unitary matrices and $\Sigma$ is a diagonal matrix with the singular values of $A$ arranged in non increasing order. Define the following partial isometry $$S_p=V\begin{pmatrix} I_p & 0 \\ 0 & 0 \end{pmatrix}W^*,$$ where $I_p$ is the $p \times p$ identity matrix. 
The following estimate was proved in \cite[Thm. 3.7]{AC16}:
$$
\| A- S_p\|_2 \leq \| A- Y\|_2
$$
for any $m \times n$ partial isometry $Y$. 
\end{rem}

Now we can give our result on existence and uniqueness of global symmetric approximations. 

\begin{teo}\label{global thm}
Let $\cF=\{ f_i \}$ be a frame in a subspace $\cK\subseteq \cH$,  and let $\{ u_i\}$ be its canonical Parseval frame. Suppose that there exists a Parseval frame quadratically close to $\{ f_i\}$. Let $F$ be the synthesis operator associated to $\{ f_i\}$,  $F=U|F|$ its polar decomposition and $\{s_i\}_1^\alpha$ its nonzero singular values. Define $r:=\# \{ \, j \, : \, s_j \leq 1/2\,\}$, and consider $E$ the orthogonal projection onto   $\oplus_{j=1}^r \ker(|F|-s_{n_j}I)$, where $s_{n_1} \leq \ldots \leq s_{n_r}\leq 1/2$. 
Then the following estimate
$$
\sum_{i=1}^\infty \| f_i - (u_i - UEe_i) \|^2 \leq  \sum_{i=1}^\infty \| f_i - x_i \|^2
$$
is valid for all Parseval frames  $\{ x_i \} \in \cQ_\cF$. Moreover, the Parseval frame $\{u_i - UEe_i \}$ is the unique symmetric approximation in $\cQ_\cF$ if and only if $\#\{ \, j \, : \, s_j=1/2\, \}\leq 1$.  If this condition fails, then there are infinitely many symmetric approximations.
\end{teo}


\begin{proof}
 We begin by noting that we have assumed $\cQ_\cF\neq\emptyset$, so it makes sense to look for  symmetric approximations. Next note that $r<\infty$. This follows  by Lemma \ref{spectral},  where we have proved that $s_i \to 1$ when $\alpha=\infty$.  According to Proposition \ref{compo con}, we can write as a disjoint union of connected components $\cQ_\cF=\bigcup_{k \in \mathbb{J}}\cQ_\cF^k$, where $\mathbb{J}\subseteq \mathbb{Z}$ is an infinite subset.  
From Theorem \ref{k>0 aprox}, or Theorem \ref{canonico} when $r=0$, we know that the Parseval frame $\{ u_i - UEe_i \}$ is a symmetric approximation of $\{ f_i\}$ in $\cQ_\cF^r$, that is,
$$
d (\{ f_i\}, \cQ_\cF^r)=d(\{ f_i\},\{ u_i - UEe_i \}).
$$ 
Now we divide the proof into three cases. We will show that the $\ell^2$-distance of the frame $\{ f_i \}$ to $\cQ^r_\cF$ is lower than the $\ell^2$-distance to the other components.

\medskip

In the first case we compare with components of the form  $\cQ_\cF^k$, $k<0$. By Theorem \ref{S aprox}, the existence of symmetric approximations to these components is guaranteed. Using this fact and Remark \ref{distance}, we have $d (\{ f_i\}, \cQ_\cF^r)< d(\{ f_i\}, \cQ_\cF^k)$.

In the second case, we  compare with components  $\cQ_\cF^{k}$, $0 \leq k<r$. According to Theorem \ref{k>0 aprox}, there  exists a symmetric approximation restricted to these components by the assumption $k<r$. Using  that $a^2\leq(a-1)^2$ for $0<a\leq 1/2$ and Remark \ref{distance},  we have
$
d (\{ f_i\}, \cQ_\cF^r) \leq d (\{ f_i\}, \cQ_\cF^k)
$. Furthermore, equality holds if and only if $s_{n_{k+1}}=\ldots =s_{n_{r}}=1/2$.

In the last case, we  compare with components  $\cQ_\cF^{k}$, $k>r$. If $\alpha<\infty$, then there exist symmetric approximations, and the proof is similar to the previous cases. In the case where $\alpha=\infty$,  we might not have the existence of symmetric approximations restricted to these components. Equivalently, the distance of the frame $\{ f_i \}$ to $\cQ_\cF^{k}$ might not be attained. 
  We have $\cQ_\cF^k\neq \emptyset$, so there exist $k$ nonzero singular values $s_{n_1}\leq \ldots \leq s_{n_r}\leq 1/2 < s_{n_{r+1}} \leq \ldots \leq s_{n_k}$ of the  operator $F$.  Pick $\{ g_{n_i} \}_1^k$ orthonormal vectors satisfying $|F|g_{n_i}=s_{n_i}g_{n_i}$, $i=1, \ldots, k$.
Note that in the proof of Lemma \ref{spectral} we may complete this  basis to an orthonormal basis $\{ g_i \}$ of $\ker(F)^\perp$ such that 
$
F=\sum_{i=1}^\infty s_i (Ug_i)\otimes g_i $.

Take a Parseval frame $\{ x_i \} \in \cQ^k_\cF$ with synthesis operator $X$. Set  $Q=X^*X$,   $\{ m_1 , m_2 , \ldots \}=\N\setminus \{ n_1, \ldots, \, n_k \, \}$ and 
$$
P'=\sum_{i=1}^\infty g_{m_i}\otimes g_{m_i}\, , \, \, \, \, \, \, \, \, E'=\sum_{i=r+1}^k g_{n_i}\otimes g_{n_i}.
$$
Note that $E=\sum_{i=1}^r g_{n_i}\otimes g_{n_i}$. Since   $j(P',Q)=0$ and $P'-Q$ is Hilbert-Schmidt,  there is a unitary $L \in \cU_2(\ell^2)$ such that 
$Q=LP'L^*$. We define 
$$
P_n=\sum_{i=1}^n g_{m_i}\otimes g_{m_i}, \, \, \, \, \, \, Q_n=LP_nL^*,
$$
and 
$
U_r=U(P'+E'),
$
which is the synthesis operator of the Parseval frame $\{ u_i - UEe_i \}$.
We apply the result stated in  Remark \ref{finite dim case} with $p=n+k-r$, $A=F(P_n+E+E')$, $S_{n+k-r}=U(P_n+E')=U_r(P_n+E+E')$ and $Y=XQ_n$. Thus, we obtain
$$
\|(F-U_r)(P_n+E+ E')\|_2 \leq \| F(P_n+E+E')-XQ_n\|_2.
$$
Letting $n \to \infty$, we find that 
$
\|F-U_r\|_2 \leq \|F-X\|_2 
$ for any partial isometry $X \in \cQ_\cF^k$. Hence we obtain the desired inequality
$d (\{ f_i\}, \cQ_\cF^r) \leq d (\{ f_i\}, \cQ_\cF^k)$. From these three cases, we conclude that 
$\{ u_i - UEe_i \}$ is a symmetric approximation of $\{ f_i\}$ in $\cQ_\cF$.

\medskip

Now we turn to the uniqueness assertion. If we have $l:=\#\{ \, j \, : \, s_j=1/2\, \}> 1$, then $d (\{ f_i\}, \cQ_\cF^r)=d (\{ f_i\}, \cQ_\cF^k)$ for $k=r-l+1, \ldots, r-1$. Using the notation of Theorem \ref{k>0 aprox}, we have $\#\mathbb{A}=r-l$ and $\#\mathbb{B}=l$, so 
$k<\#\mathbb{A}+\#\mathbb{B}$, and then there are infinitely many symmetric  approximation belonging to each $\cQ_\cF^k$. 
 Conversely, suppose that $\#\{ \, j \, : \, s_j=1/2\, \}\leq 1$. Assume that $\{ x_i\}$ is another symmetric approximation in $\cQ_\cF$. Then $\{ x_i\}$ is in particular a symmetric approximation in $\cQ_\cF^k$ for some $k \in \mathbb{J}$. We have seen that $d (\{ f_i\}, \cQ_\cF^r)< d(\{ f_i\}, \cQ_\cF^k)$ for $k <0$. In the case where $k\neq r$, $k\geq 0$, the equality $d (\{ f_i\}, \cQ_\cF^r) = d(\{ f_i\}, \cQ_\cF^k)$ implies   $s_{n_{k+1}}=\ldots =s_{n_r}=1/2$ if $k<r$. On the other hand, if  $r<k$, then  $d (\{ f_i\}, \cQ_\cF^r) = d(\{ f_i\}, \cQ_\cF^k)$ is not possible, since it would give the contradiction $(k-r)/2=s_{n_{r+1}}+ \ldots + s_{n_{k}}>(k-r)/2$.
 Thus, we must have that $\{ x_i \}$ is a symmetric approximation in $\cQ_\cF^r$. But $1/2$ has multiplicity at most one as singular value of $F$, so by Theorem \ref{k>0 aprox} $ii)$, $\{ x_i \}$ is the unique symmetric approximation of $\{ f_i\}$ in $\cQ_\cF^r$.
Hence $\{ x_i \}=\{ u_i-UEe_i \}$. 
\end{proof}

\begin{rem}
We emphasize that if $r=0$, then $E=0$, and thus the canonical Parseval frame $\{ u_i\}$ associated to $\{ f_i\}$ is the unique global symmetric approximation. 
On the other hand, observe that in the last part of the proof we actually exhibit all the possible symmetric approximations when $l=\#\{ \, j \, : \, s_j=1/2\, \} > 1$. In addition to the global symmetric approximation in the statement, there are infinitely many global symmetric approximations belonging to each component $\cQ_\cF^k$, $k=r-l+1, \ldots, r-1$, which have the form given in Theorem \ref{k>0 aprox}.
\end{rem}


\begin{rem}
Symmetric approximation can be thought as a problem of best approximation of closed range operators by partial isometries in  the Hilbert-Schmidt norm. We point out that all the results concerning the existence of best approximation proved in this paper can be generalized to the setting of symmetrically-normed ideals. However,  uniqueness results can be extended only to the $p$-Schatten ideals ($1<p<\infty$), where Lemma \ref{maher} still holds (see \cite[Thm. 4.1]{M89}). 
\end{rem}

\medskip

\subsection*{Acknowledgment}
I am grateful to  Jorge Antezana  for  several helpful discussions.   
This research was supported by Grants CONICET (PIP 2016 112201), FCE-UNLP (11X681) and ANPCyT (2015 1505).


                                               

\bigskip

{\sc (Eduardo Chiumiento)} {Departamento de de Matem\'atica, FCE-UNLP, Calles 50 y 115, 
(1900) La Plata, Argentina  and Instituto Argentino de Matem\'atica, `Alberto P. Calder\'on', CONICET, Saavedra 15 3er. piso,
(1083) Buenos Aires, Argentina.}     
                                               
\noi e-mail: {\sf eduardo@mate.unlp.edu.ar}   


\begin{thebibliography}{XX}


\bibitem{AEG80} J.G. Aiken, J.A. Erdos, J.A. Goldstein, {\it On L\"owdin orthogonalization}, Internat.
J. Quantum Chem. 18(1980), 1101-1108.


\bibitem{AS94} W.O. Amrein, K.B. Sinha, {\it On pairs of projections in a Hilbert space}, Linear Algebra Appl. 208/209 (1994). 425-435.




  \bibitem{AL08} E. Andruchow, G. Larotonda, {\it Hopf-Rinow Theorem in the Sato Grassmanian}, J. Funct. Anal. 255 (2008), no. 7, 1692-1712.
	
 \bibitem{AC16} J. Antezana, E. Chiumiento, {\it Approximation by partial isometries and symmetric approximation of finite frames}, J. Fourier Anal. Appl. (2017). doi:10.1007/s00041-017-9547-5.





\bibitem{ASS94} J. Avron, R. Seiler, B. Simon, {\it The index of a pair of projections}, J. Funct. Anal. 120 (1994), no. 1, 220-237. 


\bibitem{B97} R. Bhatia, {\it Matrix Analysis}, Berlin-Heildelberg-New York, Springer (1997).


\bibitem{C85} A.L. Carey,  {\it Some homogeneous spaces and representations of the Hilbert Lie group $U(\cH)_2$}, Rev. Roumaine Math. Pures Appl. 30 (1985), no. 7, 505-520.

\bibitem{C00} P.G. Casazza, {\it The art of frame theory}. Taiwan. J. Math. 4 (2000), 129-201. 

\bibitem{CK12}  P.G. Casazza, G. Kutyniok (eds.), {\it Finite frames: theory and applications}, Birkhäuser, Boston (2012).


\bibitem{K1} A. Chebira, J. Kova\v{c}evi\'c,  {\it Life beyond bases: The advent of frames (Part I)}, IEEE Signal Process. Mag. 24 (2007), no. 4,  86-104.

\bibitem{K2}  A. Chebira, J. Kova\v{c}evi\'c, {\it Life beyond bases: The advent of frames (Part II)}, IEEE Signal Process. Mag. 24 (2007), no. 5,  115-125.




\bibitem{Ch10} E. Chiumiento, {\it Geometry of $\mathfrak{I}$-Stiefel manifolds}, Proc. Amer. Math. Soc. 138 (2010), no. 1,  341-353.


\bibitem{Ch95} O. Christensen, {\it Frame perturbations}, Proc. Amer. Math. Soc. 123 (1995), no. 4, 1217-1220.



\bibitem{Chrbook} O. Christensen, {\it An introduction to frames and Riesz bases},
Birkh\"auser, Boston,  (2003).


\bibitem{AF97}  C.E. D'Attellis, E.M. Fern\'andez-Berdaguer (eds.), {\it Wavelet Theory and Harmonic Analysis in Applied Sciences}, Birkh\"{a}user, Boston, MA, 1997.

\bibitem{DS52} R.J. Duffin, A.C. Schaeffer, {\it A class of nonharmonic Fourier series}, Trans. Amer. Math. Soc. 72 (1952), 341-366.


\bibitem{EY36} C. Eckart, G. Young, {\it The approximation of one matrix by another of lower rank}, Psychometrika 
1 (1936),  211-218.


\bibitem{KLF09} H. Feichtinger, W. Kozek, F. Luef,  {\it Gabor analysis over finite abelian groups}, Appl. Comput. Harmon. Anal. 26 (2009), no. 2, 230-248.

\bibitem{FPT02} M. Frank, V. Paulsen, T. Tiballi, {\it Symmetric approximation
of frames and bases in Hilbert Spaces}, Trans. Amer. Math. Soc. 354 (2002), 777-793.

\bibitem{GK60} I.C. Gohberg,  M.G. Krein, {\it Introduction to the theory of linear non-self-adjoint operators}, Amer. Math. Soc., Providence, R.I., 1960.


\bibitem{H1} D. Han, {\it Approximations for Gabor and wavelet frames}, Trans. Amer. Math. Soc. 355 (2003),
3329-3342.

\bibitem{H2} D. Han,  {\it Tight frame approximation for multi-frames
and super-frames}, J. Approx. Theory 129 (2004), 78-93.

\bibitem{HL00} D. Han, D.R. Larson, {\it Frames, bases and group representations}. Mem. Am. Math. Soc. 147 (2000),
1-103. 


\bibitem{JS02} A. Janssen, T. Strohmer, {\it Characterization and computation of canonical tight windows for
Gabor frames}, J. Fourier Anal. Appl. 8 (2002), 1-28.

\bibitem{LM99} C.-K. Li, R. Mathias, {\it The Lidskii-Mirsky-Wielandt theorem--additive and multiplicative versions}, Numer. Math. 81 (1999), 377-413.


\bibitem{L70} P.-O. L\"owdin, {\it On the nonorthogonality problem}, Adv. Quantum Chem. 5 (1970), 185-199.


\bibitem{M89} P.J. Maher, {\it Partially isometric approximation of positive operators}, Illinois J. Math. 33
(1989),  227-243.

\bibitem{MBM10}  S.K. Mitra,  P. Bhimasankaram,  S.B. Malik, {\it Matrix Partial
Orders, Shorted Operators and Applications}, World Scientific, New Jersey, 2010. 

\bibitem{SV78} \c{S}. Str$\breve{\text{a}}$til$\breve{\text{a}}$, D. Voiculescu,  {\it On a Class of KMS States for the Group U($\infty$)}, Math. Ann. 235 (1978), no. 1, 87-110.

\bibitem{W86} P.Y. Wu, {\it Approximation by partial isometries}, Proc. Edinb. Math. Soc. 29 (1986),
 255-261.


\end{thebibliography}
\end{document}